\documentclass[12pt]{amsart}
\usepackage{amsfonts,amsthm,amsmath,amssymb,amscd}
\usepackage[curve,matrix,arrow]{xy}
\input xy \xyoption{all}
\CompileMatrices

\newtheorem{thm}{Theorem}[section]
\newtheorem{prop}[thm]{Proposition}
\newtheorem{lemma}[thm]{Lemma}
\newtheorem{cor}[thm]{Corollary}

\newtheorem{remark}[thm]{Remark}
\newtheorem{defin}[thm]{Definition}

\def\A{\mathbb{A}}
\def\C{\mathbb{C}}

\def\R{\mathbb{R}}
\def\Z{\mathbb{Z}}
\def\P{\mathbb{P}}

\def\F{\mathbb{F}}
\def\L{\mathbb{L}}

\def\g{\mathfrak{g}}
\def\b{\mathfrak{b}}

\def\R{\mathcal{R}}
\def\W{\mathcal{W}}

\def\a{\alpha}
\def\b{\beta}
\def\g{\gamma}

\def\phi{\varphi}

\def\l{\lambda}
\def\om{\omega}

\def\Oc{{\mathcal O}}

\def\dim{{\rm dim}}
\def\deg{{\rm deg}}

\title{Schubert calculus for algebraic cobordism}
\author{Jens Hornbostel and Valentina Kiritchenko}
\thanks{The second author would like to thank Jacobs University Bremen,
the Hausdorff Center for Mathematics and the Max Planck Institute for
Mathematics in Bonn for hospitality and support. The second author was also partially
supported by the Dynasty Foundation fellowship and RFBR grant 10-01-00540-a.}

\begin{document}

\begin{abstract}

We establish a Schubert calculus for Bott-Samelson resolutions
in the algebraic cobordism ring of a complete flag variety $G/B$ extending
the results of Bressler--Evens \cite{BE} to the algebro-geometric setting .

\end{abstract}

\maketitle

\section{Introduction}

We fix a base field $k$ of characteristic $0$.
Algebraic cobordism $\Omega^*(-)$ has been invented some years
ago by Levine and Morel \cite{LM} as the universal oriented algebraic
cohomology theory on smooth varieties over $k$. In particular, its
coefficient ring $\Omega^*(k)$ is isomorphic to the Lazard ring
$\L$ (introduced in \cite{La}). In a recent article \cite{LP}, Levine and
Pandharipande
show that algebraic cobordism $\Omega^n(X)$
allows a presentation with generators being projective morphisms
$Y \to X$ of relative codimension $n(:= \dim(X)-\dim(Y))$
between smooth varieties
and relations given by a refinement
of the naive algebraic cobordism relation (involving double
point relations).
A recent result of Levine \cite{Le} which relies on
unpublished work of Hopkins and Morel asserts an isomorphism
$\Omega^n(-)\cong MGL^{2n,n}(-)$ between Levine-Morel
and Voevodsky algebraic cobordism for smooth quasiprojective
varieties. In particular, algebraic cobordism is representable
in the motivic stable homotopy category.

In short, algebraic cobordism is to algebraic
varieties what complex cobordism $MU^*(-)$ is to topological
manifolds.

\medskip

The above fundamental results being established, it is
high time for computations, which have been carried out
only in a very small number of cases (see e.g. \cite{VY}
and \cite{Ya}).
The present article focuses
on cellular varieties $X$, for which the additive structure
of $\Omega^*(X)$ is easy to describe: it is
the free $\L$-module generated by the cells
(see the next section for more precise definitions,
statements, proofs and references). So additively,
algebraic cobordism for cellular varieties behaves exactly
as Chow groups do. Of course, algebraic K-theory
also behaves in a similar way, but we will
restrict our comparisons here and below to Chow groups.
There is a ring homomorphism $\Omega^*(X) \to MU^{2*}(X(\C)^{an})$
which for cellular varieties is an isomorphism,
see Section 2.2 and the appendix.
However, computations in $\Omega^*(X)$ become more transparent and suitable
for algebro-geometric applications if
they are done by algebro-geometric methods
rather than by a translation of the already existing results for
$MU^{2*}(X(\C)^{an})$ (e. g. those of Bressler and Evens,
see \cite{BE} and below), especially if the latter
were obtained by topological methods which do not have
counterparts in algebraic geometry.

Let us concentrate on complete flag varieties $X=G/B$ where $B$ is a Borel
subgroup of a connected split reductive group $G$ over $k$.  In the case where
$G=GL_n(k)$, the cobordism
ring $\Omega^*(X)$ may be described as the quotient of a free polynomial ring
over $\L$ with generators $x_i$ being the first Chern classes of certain
line bundles on $X$ and explicit relations.
More precisely, we show (see Theorem \ref{Borel}):
\begin{thm}
The cobordism ring $\Omega^*(X)$ is isomorphic to the graded ring $\L[x_1,\ldots,x_n]$
of polynomials with coefficients in the Lazard ring $\L$
and $\deg~x_i=1$, quotient by the ideal $S$ generated
by the homogeneous
symmetric polynomials of strictly positive degree:
$$\Omega^*(X)\simeq \L[x_1,\ldots,x_n]/S.$$
\end{thm}
This generalizes a theorem
of Borel \cite{Bo} on the Chow ring (or equivalently the singular
cohomology ring) of a flag variety to its
algebraic cobordism ring.

\medskip

The Chow ring of the flag variety has a natural basis given by the
{\em Schubert cycles}.
The central problem in Schubert calculus was to find polynomials
(later called Schubert polynomials)
representing the Schubert cycles in the Borel presentation. This problem was solved
independently by Bernstein--Gelfand--Gelfand \cite{BGG}  and Demazure \cite{De} using
{\em divided difference operators} on the Chow ring (most of the ingredients were already contained in a
manuscript of Chevalley \cite{Che}, which for many years
remained unpublished). Explicit formulas for Schubert polynomials give an algorithm
for decomposing the product of any two Schubert cycles into a linear combination of
other Schubert cycles with integer coefficients.

The complex (as well as the algebraic) cobordism ring of the flag variety also has a
natural generating set given by the {\em Bott-Samelson resolutions} of the  Schubert
cycles (note that the latter are not always smooth and so, in general, do not define
any cobordism classes). For the complex cobordism ring, Bressler and Evens described
the cobordism classes of Bott-Samelson resolutions in the Borel presentation using
{\em generalized divided difference operators} on the cobordism ring \cite{BE2,BE}
(we thank Burt Totaro from whom we first learned about this reference).
Their formulas for these operators are not algebraic and  involve a passage to the classifying
space of a compact torus in $G$ and homotopy theoretic considerations
(see \cite[Corollary-Definition 1.9, Remark 1.11]{BE2} and  \cite[Proposition 3]{BE}).
One of the goals of the present paper is to prove an algebraic formula for
the generalized divided difference operators (see Definition \ref{d.operator} and Corollary
\ref{c.Gysin}).
This formula in turn implies explicit purely algebraic formulas for the polynomials
(now with coefficients  in the Lazard ring $\L$) representing the classes of
Bott-Samelson resolutions. Note that each such polynomial contains the respective
Schubert polynomial as the lowest degree term (but in most cases also has non-trivial
higher order terms). We also give an algorithm for decomposing the product of
two Bott-Samelson resolutions into a linear combination of other Bott-Samelson
resolutions with coefficients in $\L$.

We now formulate our main theorem (compare Theorem \ref{main}), which can be viewed
as an algebro-geometric analogue of the results of Bressler-Evens
\cite[Corollary 1, Proposition 3]{BE}. Let $I=(\a_1,\ldots,\a_l)$ be an $l$-tuple of
simple roots of $G$, and $R_I$ the corresponding Bott-Samelson resolution of the Schubert
cycle $X_I$ (see Section 3 for the precise definitions). Recall that there
is an isomorphism between the Picard group of the flag variety
and  the weight lattice of $G$ such that very ample line bundles map to
strictly dominant weights (see, for instance, \cite[1.4.3]{Brion}).
We denote by $L(\l)$ the line bundle on $X$ corresponding to a weight $\l$, and
by $c_1(L(\l))$ its first Chern class in algebraic cobordism.
For each $\a_i$, we define the operator $A_i$ on $\Omega^*(X)$ in a purely algebraic
way (see Section \ref{ss.operators} for the rigorous definition for arbitrary reductive groups).
Informally, the operator $A_i$
can be  defined in the case $G=GL_n$ by the formula
$$A_i=(1+\sigma_{\a_i})\frac1{c_1(L(\a_i))},$$
where $\sigma_{\a_i}$ acts on the variables $(x_1,\ldots,x_n)$ by the transposition
corresponding to $\a_i$. Here we use that the Weyl group of $GL_n$ can be identified
with the symmetric group $S_n$ so that the simple reflections $s_{\a_i}$ correspond to
elementary transpositions (see Section 2 for more details).
Note that the $c_1(L(\a_i))$ can be written
explicitly as polynomials in $x_1$,\ldots,$x_n$ using the formal group law (see Section 2).
\begin{thm}\label{t.intro2}
For any complete flag variety $X=G/B$ and any
tuple $I=(\a_1,\ldots,\a_l)$ of simple roots of $G$,
the class of the Bott-Samelson resolution $R_I$ in the
algebraic cobordism ring $\Omega^*(X)$
is equal to
$$A_{l}\ldots A_{1}R_e, $$
where $R_e$ is the class of a point.
\end{thm}

This theorem reduces the computation of the products
of the geometric Bott-Samelson classes to the  products in the polynomial ring
given by the previous theorem. Note that in the cohomology case analogously
defined operators $A_i$ coincide with the {\em divided differences operators} defined in
\cite{BGG,De}, so our theorem generalizes \cite[Theorem 4.1]{BGG} and \cite[Theorem 4.1]{De}
for Schubert cycles in cohomology and Chow ring, respectively, to Bott-Samelson
classes in algebraic cobordism.

Note that in the case of Chow ring, the theorem analogous to Theorem \ref{t.intro2}
has two different proofs. A more algebraic proof using the Chevalley-Pieri formula
was given by Bernstein--Gelfand--Gelfand (\cite[Theorem 4.1]{BGG}, see also Section
\ref{s.CP} for a short overview). Demazure gave a more geometric proof by identifying
the divided difference operators with the push-forward morphism for certain Chow rings
(\cite[Theorem 4.1]{De}, see also Section \ref{s.main}).
At first glance, it seems that the former proof is easier to extend to the algebraic
cobordism. Indeed, we were able to extend the main ingredient of this proof, namely,
the algebraic Chevalley--Pieri formula (see Proposition \ref{p.Chevalley}).
However, the rest of the Bernstein--Gelfand--Gelfand argument fails for
cobordism (see Section \ref{s.CP} for more details) while the more geometric argument
of Demazure can be extended to cobordism with some extra work. For the complex cobordism
ring this was done by Bressler and Evens \cite{BE2,BE}. To describe the push-forward
morphism they used results from homotopy theory, which are not (yet) applicable
to algebraic cobordism.
In our article, we also follow Demazure's approach. A key ingredient for extending
this approach to algebraic cobordism is a formula for the push-forward in algebraic
cobordism for projective line fibrations due to Vishik, see Proposition \ref{p.general}.
We provide a new proof of this formula using the {\em double point relation} in
cobordism introduced by Levine and Pandharipande \cite{LP}.
In general, push-forwards (sometimes also called ``transfers'' or ``Gysin homomorphisms'') for algebraic
cobordism are considerably more intricate than the ones for Chow groups. Consequently,
their computation, which applies to any orientable cohomology theory, is more
complicated.

Using the ring isomorphism $\Omega^*(X)\simeq MU^{2*}(X(\C)^{an})$
for cellular varieties, it seems possible
to deduce our Theorem \ref{t.intro2}
from the results of Bressler--Evens \cite{BE2,BE}
on complex cobordism (the main task would be to compare our algebraically defined
operators $A_i$ with theirs).  We will not exploit this approach.
Instead, all our proofs are purely algebraic or
algebro-geometric. Conversely, we note that all our  proofs concerning
algebraic cobordism ring of the flag variety (such as the proof of
Proposition \ref{p.Chevalley}) may be easily translated to proofs
for the analogue statements concerning the complex cobordism ring.
\medskip

The article \cite{BE} does
not contain any computations. It would be interesting
to do some computation using their algorithm
and then compare them with our approach, which we consider to
be the easier one due to our explicit formula for the product of
a Bott-Samelson class with the first Chern class (see formula 5.1) based
on our algebraic Chevalley-Pieri formula. (Note also that the notations
of \cite{BE} are essentially consistent with \cite{BGG},
but not always with \cite{Ma}. We rather
stick to the former than to the latter.)

\medskip

This paper is organized as follows. In the next section, we give some
further background
on algebraic cobordism, in particular, the formula for the push-forward
mentioned above. In the case of the flag variety for $GL_n$,
we describe the multiplicative
structure of its algebraic cobordism ring.
In the third section, we recall the definition of Bott-Samelson resolutions and
then express the classes of Bott-Samelson
resolutions as polynomials with coefficients in the Lazard ring.
Section 4 contains an algebraic Chevalley-Pieri
formula and a short discussion of why the proof of \cite{BGG}
for singular cohomology does not carry over to algebraic
cobordism.
The final section contains an
algorithm for computing the products of Bott-Samelson classes in terms of other
Bott-Samelson classes as well as some examples and explicit computations.

\medskip

Our main results are valid for the flag variety of an arbitrary reductive group $G$,
but can be made more explicit in the case $G=GL_n$ using Borel presentation given by
Theorem \ref{Borel}. So we will use the flag variety for $GL_n$ as the main illustrating
example whenever possible.
One might conjecture that the algebraic cobordism rings of
flag varieties with respect to other reductive groups $G$ also
allow a Borel presentation as polynomial rings over $\L$
in certain first Chern classes
modulo the polynomials fixed by the appropriate Weyl groups
(at least when passing to rational coefficients),
because the corresponding statement is valid
for singular cohomology resp. Chow groups
(compare \cite{Bo} resp. \cite{Dem}).

\medskip

After most of our preprint was finished, we learned that
Calm\`es, Petrov and Zainoulline are also
working on Schubert calculus for algebraic cobordism. It will be
interesting to compare their results and proofs to ours (their preprint is now
available, see \cite{CPZ}).
\medskip

We are grateful to Paul Bressler and Nicolas Perrin for useful discussions
and to Michel Brion and the referee for valuable comments on earlier versions of this article.

\section{Algebraic cobordism groups, push-forwards and cellular varieties}

 We briefly recall
the geometric definition of algebraic cobordism \cite{LP} and some of its basic
properties as established in \cite{LM}. For more details see \cite{LM,LP}.
Recall that (up to sign) any element
in the algebraic cobordism group $\Omega^n(X)$ for a scheme $X$
(separated, of finite type over $k$)
may be represented by a projective morphism $Y \to X$
with $Y$ smooth and $n=\dim(X)-\dim(Y)$, the relations
being the ``double point relations'',
which we explain further below.
In particular, $\Omega^*(X)$ only lives in degrees $\le \dim X$,
which we will use several times throughout the paper. Similar to the Chow
ring $CH^*$, algebraic cobordism
$\Omega^*$ is a functor on the category
of smooth varieties over $k$, covariant for projective and
contravariant for smooth and more generally lci morphisms,
which allows a theory of Chern classes. However,
the map from the Picard group of a smooth variety $X$
to $\Omega^1(X)$ given by the first
Chern class is neither a bijection nor a homomorphism anymore
(unlike the corresponding map in the Chow ring case).
Its failure of being a group homomorphism is encoded
in a {\em formal group law} that can be constructed from $\Omega^*$.
More precisely, any algebraic {\it orientable} cohomology theory
allows by definition a calculus of Chern classes,
and consequently the construction of a formal group
law. A formal group law is a formal power series $F(x,y)$ in two variables
such that for any two line bundles $L_1$ and $L_2$ we have the following identity
relating their first Chern classes:
$$c_1(L_1\otimes L_2)=F(c_1(L_1),c_1(L_2)).$$
E. g. the formal group law for $CH^*$ is additive, that is, $F(x+y)=x+y$.
Algebraic cobordism is the universal one among the algebraic orientable cohomology
theories. In what follows, $F(x,y)$ will always denote the universal formal group law
corresponding to algebraic cobordism unless stated otherwise.

In this and in many other ways - as the computations below
will illustrate - algebraic cobordism is a refinement
of Chow ring, and one has a natural isomorphism
of functors $\Omega^*(-) \otimes_{\L} \Z \cong CH^*(-)$
(see \cite{LM} where all these results are proved).
Here and in the sequel, $\L$ denotes the
Lazard ring, which classifies one-dimensional commutative
formal group laws and is isomorphic to the graded polynomial
ring $\Z[a_1,a_2,\ldots]$ in countably many variables
\cite{La}, where we put $a_i$ in degree $-i$.
When considering polynomials $p(x_1,....x_n)$
over $\L$ with $\deg(x_i)=1$, we will distinguish
the (total) {\it degree} and the {\it polynomial degree}
of $p(x_1,...,x_n)$.

Note that the Lazard ring is isomorphic to the algebraic (as well as complex)
cobordism ring of a point. In particular, its elements can be represented by the
cobordism classes of smooth varieties. In what follows, we use this geometric
interpretation.
\medskip

We are also going to use a geometric interpretation
of the formal group law, namely, the {\em double point relation}.
This is an equality for elements in the algebraic cobordism ring established
in \cite{LP}. We recall the definition for the reader's convenience.

\medskip

{\bf Double point relation:}

Assume that we have three smooth hypersurfaces $A$, $B$ and $C$ on a smooth
variety $Z$ such that the following conditions hold
\begin{enumerate}

\item $C$ is linearly equivalent to $A+B$
\item $A$, $B$ and $C$ have transverse pairwise intersections
\item $C$ does not intersect $A\cap B$
\end{enumerate}
\noindent
Then we have the following {\em double point relation}.
Denote by $D$ the intersection
$A\cap B$. We have
$$[C\to Z]=[A\to Z]+[B\to Z]-[\P_D\to Z]$$
in $\Omega^*(X)$, where $\P_D=\P(\Oc_D\oplus N_{A/D})=\P_D(N_{B/D}\oplus\Oc_D)$ and
the map $\P_D\to Z$ is the composition of the natural projection $\P_D\to D$ with the
embedding $D\subset Z$. Here $N_{A/D}$ and $N_{B/D}$ are the normal bundles to $D$ in
$A$ and $B$, respectively. The second condition ensures that
$\P(\Oc_D\oplus N_{A/D})=\P_D(N_{B/D}\oplus\Oc_D)$
(since $L(C)|_D=(L(A)\otimes L(B))|_D=N_{A/D}\otimes N_{B/D}$ is trivial).

This formulation is a special case of the extended double point relation in
\cite[Lemma 5.2]{LP}.
The double point relation allows to express geometrically the discrepancy
between the additive formal group law and the universal one. Namely, since
$C=F(A,B)$ by the first condition, we get
$$A+B-F(A,B)=[\P_D\to Z].$$
We will use this equation when proving Proposition \ref{p.general}.

\medskip

We will also use repeatedly the projective bundle formula, which we recall below for
the reader's convenience.
For more details see \cite[Section 1.1]{LM} and \cite[3.5.2]{Ma}.

\medskip

{\bf Projective bundle formula}: Let $E\to X$ be a vector bundle of rank $r$ over $X$.
Denote by
$Y=\P(E^*)$ the variety of hyperplanes  of $E$, and by $\pi$ the natural projection
$\pi:Y\to X$. The variety  $\P(E^*)$ is a fibration over $X$ with
fibers isomorphic to $\P^{r-1}$. Note that equivalently $\P(E^*)$ can be defined as
the variety of one-dimensional
quotients of $E$ since there is a canonical isomorphism between the variety of
hyperplanes and the variety of quotients by hyperplanes in a vector space.
This is how $\P(E^*)$ is defined in  \cite[Section 1.1]{LM}  (where it is denoted by
$\P(E)$). Let $A^*(-)$ be any oriented cohomology
theory. Denote by $\xi$ the first
Chern class of the tautological quotient
line bundle $\Oc_E(1)$ on $Y$ whose restriction on each fiber of $Y$ over $X$ coincides
with $\Oc_{\P^{r-1}}(1)$. The first Chern can be defined as $\xi=s^*s_*(1_Y)$ where
$s:Y\to\Oc_E(1)$ is the zero section and
$1_Y\in A^0(Y)$ is the multiplicative unit element.  Then there is a ring isomorphism:
$$A^*(Y)=A^*(X)[\xi]/(\sum_{j=0}^r(-1)^jc_j(\pi^*E)\xi^{r-j}).$$
The isomorphism identifies a polynomial $b_0+b_1\xi+\ldots+b_{n-1}\xi^{n-1}$
in $A^*(X)[\xi]$ with the element
$\pi^*b_0+(\pi^*b_1)\xi+\ldots+(\pi^*b_{n-1})\xi^{n-1}$ in $A^*(Y)$. In particular,
$A^*(Y)$ splits into the direct sum
$\pi^*A^*(X)\oplus \xi \pi^*A^*(X)\oplus\ldots\oplus \xi^{n-1}\pi^*A^*(X)$.

Note that the relation
$$\sum_{j=0}^r(-1)^jc_j(\pi^*E)\xi^{r-j}=0$$
admits the following alternative description.
Consider a short exact sequence of vector bundles on $Y$:
$$0\to\tau_E\to \pi^*E\to\Oc_E(1)\to 0,$$ where $\tau_E$ is the tautological hyperplane
bundle on $Y$. By the Whitney sum formula we have that the
total Chern class $c(\pi^*E)$ is equal to the product $c(\tau_E)c(\Oc_E(1))$. Since
$c(\Oc_E(1))=1+\xi$
we have $c(\pi^*E)=c(\tau_E)(1+\xi)$. We now divide this identity by $(1+\xi)$
(that is, multiply by
$\sum_{j=0}^{r+\dim X-1} (-1)^{j}\xi^j$) and get that
$c(\tau_E)=c(\pi^*E)(\sum_{j=0}^{r+\dim X-1} (-1)^{j}\xi^j)$. In particular,
$$c_{r}(\tau_E)=(-1)^r\sum_{j=0}^r(-1)^jc_j(\pi^*E)\xi^{r-j},$$
so we can interpret the relation above as the vanishing of the $r$-th Chern class
of the bundle $\tau_E$ (which has rank $r-1$).

\subsection{A formula for the push-forward}\label{s.Gysin}

Let $X$ be a smooth algebraic variety,
and $E\to X$ a vector bundle of rank two on $X$.
Consider the projective line fibration $Y=\P(E)$ defined as the variety of all lines
in $E$. We have a natural projection $\pi:Y\to  X$ which is projective
and hence induces a
{\em push-forward} (or {\em transfer}, sometimes also called {\em Gysin map})
$\pi_*:\Omega^*(Y)\to\Omega^*(X)$.
We now state
a formula for this push-forward. Note that
this formula is true not only for algebraic cobordism but for any
orientable cohomology theory, as the proofs remain true in this more general
case.

Consider the ring  of formal power series in two variables $y_1$ and $y_2$ with
coefficients in $\Omega^*(X)$. Define the operator $A$ on this ring by the formula
$$A(f)=(1+\sigma)\frac{f}{F(y_1,\chi(y_2))},$$
where $[\sigma(f)](y_1,y_2):=f(y_2,y_1)$.
Here $F$ is the universal formal group law (or more generally, the one
of the orientable cohomology theory one considers) and $\chi$ is the inverse for the formal group
law $F$, that is, $\chi$ is uniquely determined by the equation $F(x,\chi(x))=0$
(we use notation from \cite[2.5]{LM}).
The operator $A$ is an analog of the {\em divided difference operator} introduced
in \cite{BGG,De}. In the case of Chow rings, our definition coincides with the classical
divided difference operator, since the formal group law for Chow ring is additive,
that is, $F(x,y)=x+y$ and $\chi(x)=-x$.
Though $A(f)$ is defined as a fraction, it is easy to write it as a formal power series
as well (see Section 5). Such a power series is unique since
$F(y_1,\chi(y_2))=y_1-y_2+\ldots$
is clearly not a zero divisor.
E.g. we have
$$A(1)=\frac{x+\chi(x)}{x\chi(x)}=q(x,\chi(x))=
-a_{11}-a_{12}(x+\chi(x))+\ldots,$$
where $x=F(y_1,\chi(y_2))$, and $q(x,y)$ is the power series uniquely determined by the
equation $F(x,y)=x+y-xyq(x,y)$. In particular, since $F(x,\chi(x))=0$ by definition of the power series
$\chi(x)$, we have $x+\chi(x)-x\chi(x)q(x,\chi(x))=0$ which justifies the second equality.
For the last equality, we used computation of the first few terms of
$F(x,y)$ and $\chi(x)$ from \cite[2.5]{LM}. Here $a_{11}$, $a_{12}$ etc.  denote the
coefficients of the universal formal group law, that is,
$$F(x,y)=x+y+a_{11}xy+a_{12}xy^2+\ldots.$$
The coefficients $a_{ij}$ are the elements of the Lazard ring $\L^*$, e.g.
$a_{11}=-[\P^1]$, $a_{12}=a_{21}=[\P^1]^2-\P^2$ (see \cite[2.5]{LM}).
We also have
$$A(y_1)=y_2A(1)+\frac{F(x,y_2)-y_2}{x}=
y_2q(x,\chi(x))-y_2q(x,y_2)+1=1+a_{12}y_1y_2+\ldots.$$
The pull-back $\pi^*:\Omega^*(X)\to\Omega^*(Y)$ gives $\Omega^*(Y)$
the structure of an $\Omega^*(X)$-module.
Recall that by the projective bundle formula we have an isomorphism of
$\Omega^*(X)$-modules
$$\Omega^*(Y)\cong\pi^*\Omega^*(X)\oplus \xi\pi^*\Omega^*(X),$$
where $\xi=c_1(\Oc_E(1))$.
Since the push-forward is a homomorphism of $\Omega^*(X)$-modules,
it is enough to
determine the action of $\pi_*$ on $1_Y$ and on $\xi$.
The following result is a special case of \cite[Theorem 5.30]{Vi}, which gives an explicit formula
for the push-forward $\pi_*$ for vector bundles of arbitrary rank.
\begin{prop}\cite[Theorem 5.30]{Vi}\label{p.general}
Let $\xi_1$ and $\xi_2$ be the Chern roots of $E$, that is, formal variables
satisfying the conditions $\xi_1+\xi_2=c_1(E)$ and $\xi_1\xi_2=c_2(E)$.
Then the push-forward acts on $1_Y$ and $\xi$ as follows:
$$\pi_*(1_Y)=[A(1)](\xi_1,\xi_2),$$
$$\pi_*(\xi)=[A(y_1)](\xi_1,\xi_2),$$
where $A(1)$ and $A(y_1)$ are the formal power series in two variables defined above.

Since $A(1)$ and $A(y_1)$ are symmetric in $y_1$ and $y_2$, they can be written as
power series in $y_1+y_2$ and $y_1y_2$. Hence, the right hand sides are power series
in $c_1(E)$ and $c_2(E)$ and even polynomials (as all terms of degree greater
than $\dim~X$ will vanish by \cite{LP}). So the right hand sides indeed define elements
in $\Omega^*(X)$.
\end{prop}

For the Chow ring and $K_0$, analogous statements were proved in
\cite[Propositions 2.3,2.6]{De} for certain morphisms $Y \to X$.
Note that for both of these theories, the formula for
$\pi_*(\xi)$ reduces to $\pi_*(\xi)=1$ since the corresponding formal group
laws do not contain terms of degree greater than two.
As Vishik showed (see \cite[Theorem 5.35]{Vi}), his formula is equivalent to
Quillen's formula \cite{Q} for complex cobordism, as also proved by Shinder
in the algebraic setting \cite{Sh}. We give a new geometric proof of Proposition
\ref{p.general} (that is, of Vishik's formula for rank two bundles)
based on the double point  relation in algebraic cobordism.

\begin{proof} First, note that replacing $E$ with $E_M=E\otimes M$ for an
arbitrary line bundle $M$ on $X$ does not
change the variety $Y$ and the map $\pi$. However, this does change the tautological
quotient line $c_1(\Oc_E(1))$. More precisely, we have the following isomorphism
of line bundles on $Y$ (compare e.g. \cite[Proof of Lemma 7.1]{LP}):
$$\pi^*M\otimes\Oc_E(1)=\Oc_{E_M}(1).$$
Let us denote by $\xi_M$ the first Chern class of the tautological quotient
line bundle $\Oc_{E_M}(1)$. The identity above implies that
$\xi_M=F(\xi,\pi^*c_1(M))$ or equivalently $\xi=F(\xi_M,\pi^*c_1(M^*))$.
Hence, to compute $\pi_*\xi$ it is enough to compute
$\pi_*1_Y$ and $\pi_*\xi_M$ for some $M$.
It is convenient to choose $M=L_1^*$ so that
$E_M$ has a trivial summand, and hence one of the
Chern roots of $E_M$ is zero. Thus we can assume that $E=\Oc_X\oplus L$.
In this case, the second formula of Proposition \ref{p.general} reduces to
$\pi_*\xi=1$, which is easy to show by similar methods as for the Chow ring.
Namely, consider the natural embedding $i:X=\P(\Oc_X)\to Y=\P(E)$.
Then $\xi=i_*1_X$ by \cite[Lemma 5.1.11]{LM}. Hence,
$\pi_*(\xi)=\pi_*i_*1_X=1_X$ since $\pi\circ i=id_X$.

It is more difficult to compute
$\pi_*1_Y$, which is the cobordism class of $[\pi:Y\to X]$. For the Chow ring, it is
zero by degree reasons, but for cobordisms it is not. E.g. even for a trivial bundle
$E$ we have $[\pi:Y\to X]=[\pi:X\times\P^1\to X]=-a_{11}1_X$. We compute $[\pi:Y\to X]$
by representing it as one of the terms in a suitable double point relation. Namely,
consider the variety $Z=Y\times_X Y$ and define three smooth hypersurfaces $A$, $B$
and $C$ on $Z$ as follows: $A=\{y\times i(\pi(y)): y\in Y)\}$,
$B=\{i'(\pi(y))\times y:y\in Y\}$ and  $C=\{y\times y: y\in Y\}$.
Here $i:X\to Y$ and $i':X\to Y$ are the embeddings $\P(\Oc_X)\to \P(E)$ and
$\P(L)\to \P(E)$, respectively.
Then it is easy to check (using again \cite[Lemma 5.1.11]{LM}) that
$A=c_1(p_1^*\Oc_E(1))$, $B=c_1(p_2^*\Oc_{E\otimes L^*}(1))$ and
$C=c_1(\Oc_E(1)\boxtimes \Oc_{E\otimes L^*}(1))$, where $p_1$, $p_2$
are the projections of $Z$ onto the first and the second factor, respectively.
Hence, we have the FGL identity $C=A+B-ABq(A,B)$ on $Z$, from which we can easily
get the double point relation we need.
Namely, apply $\pi_*{p_1}_*$ to both sides and get that
$$[\pi:Y\to X]=\pi_*{p_1}_* ABq(A,B),$$
because the other two terms cancel out.
The right hand side can be computed by the projection formula using that
$AB=\sigma_*1_X$, where $\sigma: X\to Z$ sends $x$ to $i(x)\times i'(x)$.
\end{proof}

If we identify $\Omega^*(Y)$ with the polynomial ring
$\Omega^*(X)[\xi]/(\xi^2-c_1(E)\xi+c_2(E))$ by the projective bundle formula, we can
reformulate Proposition \ref{p.general} as follows:
$$\pi_*(f(\xi))=[A(f(y_1))](\xi_1,\xi_2)$$
for any polynomial $f$ with coefficients in $\Omega^*(X)$ (where $f(y_1)$ in the right
hand side is regarded as an element in $\Omega^*(X)[[y_1,y_2]]$). In this form,
Proposition \ref{p.general} is consistent with
the classical formula for the push-forward
in the case of Chow ring (cf. \cite[Remark 3.5.4]{Ma}).
Indeed, since the formal group law  for Chow ring is additive we have  $A(1)=\frac{1}{y_1-y_2}+
\frac{1}{y_2-y_1}=0$ and $A(y_1)=\frac{y_1}{y_1-y_2}+\frac{y_2}{y_2-y_1}=1$.

\begin{defin}\label{d.operator} We define an $\Omega^*(X)$-linear operator
$A_\pi$ on $\Omega^*(Y)$ as follows. We have an isomorphism
$$\Omega^*(X)[[y_1,y_2]]/(y_1+y_2-c_1(E),y_1y_2-c_2(E))\cong\Omega^*(Y)$$  given
by $f(y_1,y_2)\mapsto f(\xi,c_1(E)-\xi))$.
Then
the operator $A$ on $\Omega^*(X)[[y_1,y_2]]$ descends to an operator
$A_\pi$ on $\Omega^*(Y)$, which can be described using the above isomorphism as follows
$$A_\pi:f(\xi,c_1(E)-\xi)\to [A(f(y_1,y_2))](\xi,c_1(E)-\xi).$$

We also define a $\Omega^*(X)$-linear endomorphism $\sigma_\pi$ of $\Omega^*(Y)$ by the formula:
$$\sigma_\pi:f(\xi,c_1(E)-\xi)=f(c_1(E)-\xi,\xi).$$
\end{defin}

The operator $A_\pi$ is well-defined since $A$  preserves the ideal
$(y_1+y_2-c_1(E),y_1y_2-c_2(E))$. Indeed, for any power series $f(y_1,y_2)$
symmetric in $y_1$ and $y_2$ (in particular, for $y_1+y_2-c_1(E)$ and
$y_1y_2-c_2(E)$) and any power series $g(y_1,y_2)$
we have $A(fg)=fA(g)$.
The operator $A_\pi$ decreases degrees by one, and its image is contained in
$\pi^*\Omega^*(X)\subset \Omega^*(Y)$, which can be identified  using the above
isomorphism for $\Omega^*(X)$ with the subring of symmetric polynomials in
$y_1$ and $y_2$.
Proposition
\ref{p.general} tells us that the push-forward $\pi_*:\Omega^*(Y)\to\Omega^*(X)$ is
the composition of $A_\pi$ with the isomorphism $\pi^*\Omega^*(X)\cong\Omega^*(X)$,
which sends (under the above identifications) a symmetric polynomial $f(y_1,y_2)$ into
the polynomial $g(c_1(E),c_2(E))$ such that $g(y_1+y_2,y_1y_2)=f(y_1,y_2)$.
Hence, we get the
following corollary, which we will use in the sequel.

\begin{cor} \label{c.Gysin} The composition
$\pi^*\pi_*:\Omega^*(Y)\to\Omega^*(Y)$ is equal to the operator
$A_{\pi}$:
$$\pi^*\pi_*=A_\pi.$$
\end{cor}
In the special case $Y=G/B$ and $X=G/P_i$ (and this is the main application we have, see Section
\ref{ss.operators}), the topological analogue
of this formula appeared in \cite[Corollary-Definition 1.9]{BE2}
for a different definition of $A_{\pi}$.

\subsection{Algebraic cobordism groups of cellular varieties}
We start with the definition of a cellular variety. The following definition is taken
from \cite[Example 1.9.1]{Fu},
other authors sometimes consider slight variations.
\begin{defin}
We say that a smooth variety $X$ over $k$ is ``cellular''
or ``admits a cellular decomposition'' if $X$ has a filtration
$\emptyset=X_{-1} \subset X_0 \subset X_1 \subset ... \subset X_{n}=X$
by closed subvarieties such that the $X_i - X_{i-1}$ are isomorphic
to a disjoint union of affine spaces $\A^{d_i}$ for all $i=0,...,n$,
which are called
the ``cells'' of $X$.
\end{defin}
Examples of cellular varieties include projective spaces and more general
Grassmannians, and complete flag varieties $G/B$ where $G$ is a split reductive
group and $B$ is a Borel subgroup.

The following theorem is a corollary of
\cite[Corollary 2.9]{VY}. We thank Sascha Vishik for explaining to us
how it can be deduced using the projective
bundle formula. The main point
is that for $d=\dim~X$
and $i$ an arbitrary integer, one has for $A=\Omega$ that
$\Omega^i(X)=: \Omega_{d-i}(X)$
is isomorphic to $Hom(A(d-i)[2(d-i)],M(X))$
using that in the notation of loc. cit.
$Hom(A(d-i)[2d-2i],M(X))$ is a direct summand
in $Hom(M(\P^{d-i}),M(X))=A_{d-i}(\P^{d-i} \times X)
=\oplus_{j=0}^{d-i}A_{d-i-j}(X)$,
and it is not difficult to see that it corresponds to
the summand with $j=0$.

\begin{thm}\label{t.cellular}
Let $X$ be a variety with a cellular
decomposition as in the definition above.
Then we have an isomorphism of graded abelian groups
(and even of $\L$-modules)
$$\Omega^*(X)\cong \oplus_i \L [d_i]$$
where the sum is taken over the cells of $X$. There is a basis in $\Omega^*(X)$ given by resolutions
of cell closures (choose one resolution for each cell).
\end{thm}
The second statement of this theorem follows from the first one if we show that the cobordism classes
of resolutions of the cell closures generate $\Omega^*(X)$. This can be deduced from the analogous
statement for the Chow ring using \cite[Theorem 1.2.19, Remark 4.5.6]{LM}.
For complex cobordism of topological complex cellular spaces,
the corresponding theorem simply follows
from an iterated use of the long exact
localization sequence which always splits as
everything in sight has $MU^*$-groups concentrated
in even degrees only.
Note also that in the topological case, the Atiyah-Hirzebruch spectral
sequence degenerates for these spaces, which allows to transport information
from singular cohomology to complex cobordism. As Morel points out,
the analogous motivic spectral sequence invented by Hopkins-Morel (unpublished)
converging to algebraic cobordism does not in general degenerate even for the
point $Spec(k)$, because the one converging to algebraic $K$-theory
does not.

\medskip

We now turn to the ring structure.
First, we note that if $k=\C$, then
there is a map of graded rings
and even of $\L$-algebras
$\Omega^*(X) \to MU^{2*}(X(\C)^{an})$
by universality of algebraic cobordism \cite[Example 1.2.10]{LM}.
Using the geometric description of push-forwards
both for $\Omega^*$ and $MU^*$ and the fact
that the above morphism respects push-forwards
\cite{LM} as well as \cite{LP}, we may describe this
map explicitly by mapping an element
$[Y \to X]$ of $\Omega^*(X)$ to $[Y(\C)^{an} \to X(\C)^{an}]$
in $MU^{2*}(X(\C)^{an})$. As both product structures are defined by taking
cartesian products of the geometric representatives and pulling it
back along the diagonal of $X$ resp. $X(\C)^{an}$, we see that
this map does indeed preserve the graded $\L$-algebra structure. Also,
for any embedding $k \to \C$ we obtain a ring homomorphism
from algebraic cobordism over $k$ to algebraic cobordism
over $\C$.

For the flag variety of $GL_n$, this is an isomorphism
by Theorem 2.6 below which is also valid for $MU^*$,
as both base change from $k$ to $\C$ and complex topological
realization respect products and first Chern classes.
For general cellular varieties, it is still an isomorphism.
This is probably known to the experts, we provide a
proof in the appendix.

For some varieties $X$, the ring structure of $\Omega^*(X)$ can be completely determined using
the projective bundle formula \cite[Section 1.1]{LM}. This is the case for the variety of complete flags
for $G=GL_n$ (see Theorem \ref{Borel} below) and also for Bott-Samelson resolutions of Schubert cycles
in a complete flag variety for any reductive group $G$ (see Section 3).
\medskip

\subsection{Borel presentation for the flag variety of $GL_n$} We now turn to the case of the complete flag variety $X$ for $G=GL_n(k)$. The
points of $X$ are identified with {\em complete flags} in $k^n$. A {\em complete flag} is a strictly
increasing sequence of subspaces
$$F=\{\{0\}=F^0\subset F^1\subset F^2\subset\ldots\subset F^n=k^n\}$$
with $\dim(F^k)=k$. The group $G$ acts transitively on the set of all flags, and the stabilizer of a point
is isomorphic to a Borel subgroup $B\subset G$, which makes $X=G/B$ into a homogeneous space under $G$.
By this definition,  $X$ has structure of an algebraic variety.

Note that over $\C$, one may equivalently define the flag variety $X$ to
be the homogeneous space $K/T$
under the maximal compact subgroup $K\subset G$, where $T$ is a maximal
compact torus in $K$ (that is, the product of several copies of $S^1$) \cite{Bo}.
E. g., for $G=GL_n(\C)$ (resp. $SL_n(\C)$), the maximal compact
subgroup is $U(n)$ (resp. $SU(n)$).
This is the language in which many of the definitions and results
in \cite{BGG}, \cite {Bo} and \cite{BE} are stated.
We sometimes allow ourselves to use those definitions
and results which do carry over
to the ``algebraic" case (reductive groups over $k$)
without mentioning explicitly
the obvious changes that have to be carried out.

There are $n$ natural line bundles $L_1$,\ldots,$L_n$ on $X$, namely, the fiber of $L_i$
at the point $F$ is equal to $F^{i}/F^{i-1}$.
Put $x_i=c_1(L_i)$, where the first Chern class $c_1$ with respect to algebraic
cobordism
is defined in \cite{LM}. Note that our definition of $x_i$ differs by sign
from the one in \cite{Ma}.
The following result on the algebraic cobordism ring
is an analog of the Borel presentation
for the singular cohomology ring of a flag variety.
In fact, it holds for any orientable
cohomology theory since its proof only uses the projective bundle formula.

\begin{thm}\label{Borel} Let $A^*(-)$ be any orientable cohomology theory
(e.g. $CH^*(-)$ or $\Omega^*(-)$).
Then the ring $A^*(X)$ is isomorphic as a graded ring to the ring
of polynomials in $x_1$,\ldots, $x_n$ with coefficients in
the coefficient ring $A^*(pt)$ and $deg(x_i)=1$, quotient by the ideal $S$ generated by the
symmetric polynomials of strictly positive polynomial degree:
$$A^*(X)\simeq A^*(pt)[x_1,\ldots,x_n]/S.$$
More generally, let $E$ be a vector bundle of rank $n$
over a smooth variety $Y$
and $\F(E)$ be the flag variety relative to this bundle.
Then we have an isomorphism of graded rings
$$A^*(\F(E))\simeq A^*(pt)[x_1,...,x_n]/I$$
where $I$ is the ideal generated by the relations
$e_k(x_1,..,x_n)=c_k(E)$ for $1 \leq k \leq n$
with $e_k$ denoting the $k$-th elementary symmetric
polynomial.
\end{thm}
\begin{proof}
The proof of \cite[Theorem 3.6.15]{Ma}
for the Chow ring case can be slightly modified so that it becomes applicable to any other
orientable theory $A^*$. Namely, for an arbitrary oriented cohomology theory $A^*$,
it is more convenient to dualize the geometric argument in \cite[Theorem 3.6.15]{Ma}
because we can no longer use that $c_i(E)=(-1)^ic_i(E^*)$ for a vector bundle $E$
(which is used implicitly several times in the proof of \cite[Theorem 3.6.15]{Ma}).
That is, we start with the variety of partial flags
$P_i=\{\ F^{n-i} \subset F^{n-i+1}\subset\ldots\subset F^n=k^n \}$
(e. g. $P_1$ is the variety of hyperplanes in $k^n$
and $P_{n-1}=X$). The rest of the argument is completely analogous to the proof of
\cite[Theorem 3.6.15]{Ma}. We give the details below for the reader's convenience.

Denote by $\W_j$ the corresponding
tautological vector bundle of rank $j$ over $P_i$, where $j \ge n-i$ (that is, the fiber of $\W_j$
over a point $\{\ F^{n-i} \subset F^{n-i+1}\subset\ldots\subset F^n \}$ is equal to $F_{j}$).
In particular, $L_i$ defined above is equal to $\W_i/\W_{i-1}$.
Put $x_i=c_1(L_i)$.
As $P_i =\P((\W_{n-i+1})^*)$ is the projective bundle over $P_{i-1}$
and the line bundle $\Oc_{\W_{n-i+1}}(1)$ is isomorphic to
$L_{n-i+1}$, the projective bundle formula for
orientable cohomology theories \cite[Section 1.1]{LM} yields
$$A^*(P_i) \cong A^*(P_{i-1})[x_{n-i+1}]/(\sum_{j=0}^{n-i+1}(-1)^jc_j(\W_{n-i+1})
x_{n-i+1}^{n-i+1-j}).$$
Or, using the above interpretation of the relation in the projective bundle formula
and the short exact sequence of
vector bundles on $P_i$
$$ 0 \to \W_{n-i} \to \W_{n-i+1} \to  L_{n-i+1} \to 0$$
we get
$$A^*(P_i) \cong A^*(P_{i-1})[x_{n-i+1}]/(c_{n-i+1}(\W_{n-i})).$$

It remains to compute $c_{n-i+1}(\W_{n-i})$. This can be done by induction on $i$
starting from $i=0$ (in which case $\W_n$ is a trivial line bundle) and applying
the Whitney sum formula to the short exact sequence of vector bundles above.
We get $c(\W_{n-i})=\prod_{j=n-i+1}^nc(L_j)^{-1}=\prod_{j=n-i+1}^n(1+x_j)^{-1}=
\sum_{k\ge 0}(-1)^kh_k(x_{n-i+1},\ldots,x_n)$, where $h_k(x_{n-i+1},\ldots,x_n)$ denotes the sum of all monomials
of degree $k$ in $x_{n-i+1}$,\ldots,$x_n$. In particular,
$c_{n-i+1}(\W_{n-i})=(-1)^{n-i+1}h_{n-i+1}(x_{n-i+1},\ldots,x_n)$.
From this we deduce that
$$A^*(P_i) \cong A^*(P_{i-1})[x_{n-i+1}]/(h_{n-i+1}(x_{n-i+1},\ldots,x_n)),$$
and hence $$A^*(P_{n})\cong A^*(pt)[x_1,\ldots,x_n]/(h_n(x_n),h_{n-1}(x_{n-1},x_{n}),
\ldots,h_1(x_1,\ldots,x_n)).$$
The ideal generated by the relations $h_{n-i+1}(x_{n-i+1},\ldots,x_n)$ is exactly $S$,
which is easy to check starting with the recurrence relation
$$h_{i}(x_1,\ldots,x_n)=h_{i}(x_i,\ldots,x_n)+\sum_{j<i}x_jh_{i-1}(x_j,\ldots,x_n).$$

The proof of the more general case is completely analogous to the proof
of \cite[Proposition 3.8.1]{Ma}. Note that the proof of \cite[Proposition 3.8.1]{Ma}
does not really use \cite[Theorem 3.6.15]{Ma} as an induction base
(despite the claim in the proof) and in fact gives another proof for
\cite[Theorem 3.6.15]{Ma}, which is also applicable to an arbitrary oriented
cohomology theory.
\end{proof}
\begin{remark}\label{r.Borel} \em The proof immediately implies that the class of a point in $A^*(X)$ is
equal to $x_n^{n-1}x_{n-1}^{n-2}\cdots x_2$. Since
$x_n^{n-1}x_{n-1}^{n-2}\cdots x_2=
\frac1{n!}\prod_{i>j}(x_i-x_j)\mod S$
(which is easy to show by induction on $n$ using that
$(x_n-x_{n-1})\cdots(x_n-x_1)=nx_n^{n-1}\mod S$) we also have that the class of a point can be represented
by the polynomial $\Delta_n=\frac1{n!}\prod_{i>j}(x_i-x_j)$.

Note that the proof also gives an explicit formula for the classes of one-dimensional
{\em Schubert cycles} $X_1=X_{s_{\g_1}}, \ldots, X_{n-1}=X_{s_{\g_{n-1}}}$
in $X$
corresponding to the simple roots $\g_1$,\ldots,$\g_{n-1}$ of $GL_n$
(see the beginning of Section 3 for the definition of the Schubert cycles $X_w$ for $w$
in the Weyl group of $G$).
The cycle $X_k$ consists of flags
$F=\{\{0\}=F^0\subset F^1\subset F^2\subset\ldots\subset F^n=k^n\}$
such that all $F^i$ except for $F^k$ are fixed. Then the class of $X_k$ is equal to the
class of a point divided by $x_{k+1}$. Indeed, to get the class of $X_k\subset X$
we should take the point in $P_{n-k-1}$ corresponding to the fixed partial flag
$\{F^{k+1}\subset F^{k+2}\subset\ldots\subset F^n=k^n\}$ and then take a line
in a fiber of the projective bundle $P_{n-k}\to P_{n-k-1}$ over this point. Namely, the
line will consist of all hyperplanes in $F^{k+1}$ that contain the fixed codimension
two subspace $F^{k-1}$.
Again it is easy
to show by induction on $n$ that the polynomial
$x_n^{n-1}x_{n-1}^{n-2}\cdots x_2/x_k$ is equal to
$2\Delta_n/(x_{k+1}-x_{k})$ modulo the ideal $S$.
\end{remark}

Note that the Borel presentation for singular cohomology implies,
in particular, that Picard group of the
flag variety is generated (as an abelian group)
by the first Chern classes of the line bundles $L_1$,\ldots,$L_n$
the only nontrivial relation being
$\sum c_1(L_i)=0$.
In what
follows, we will also use the following alternative description of the Picard group of $X$. Recall that
each strictly dominant weight $\l$ of $G$ defines an irreducible representation
$\pi_\l:G\to GL(V_\l)$ and an embedding $G/B\to \P(V_\l)$. Hence, to each
strictly dominant
weight $\l$ of $G$ we can assign a very ample line bundle $L(\l)$ on $X$ by taking the pull-back
of the  line bundle $\Oc_{\P(V_\l)}(1)$ on $\P(V_\l)$. The map $\l\mapsto L(\l)$
extended to non-dominant weights by linearity gives an isomorphism
between the Picard group of $X$ and the weight lattice of $G$ \cite[1.4.3]{Brion}.
In particular, for the line bundles above we have $L_i=L(-e_i)$ where
$e_i$ is the weight of $GL_n$
given by the $i$-th entry of the diagonal torus in $GL_n$.

We now compute $c_1(L(\a_i))$ as a polynomial in $x_1$,\ldots, $x_n$.
Let $\g_1$,\ldots,$\g_{n-1}$ be the simple roots of $G$ (that is, $\g_i=e_i-e_{i+1}$).
We can express the line bundles $L(\g_i)$ in terms of the line bundles $L_1$,\ldots,$L_n$.
Since $L_i=L(-e_i)$ and $\g_i=e_i-e_{i+1}$, we have that the line bundle $L(\g_i)$ is isomorphic to
$L_i^{-1}\otimes L_{i+1}$.
In particular,  we can compute
$$c_1(L(\g_i))=c_1(L_i^{-1}\otimes L_{i+1})=F(\chi(x_i),x_{i+1}).$$
E.g. by the formulas for $F(x,y)$ and $\chi(x)$ from \cite[2.5]{LM}
the first few terms of
$c_1(L(\g_i))$ look as follows
$$c_1(L(\g_i))=-x_i+x_{i+1}+a_{11}x_i^2-a_{11}x_ix_{i+1}+\ldots,$$
where $a_{11}=-[\P^1]$.

In what follows, we will use the isomorphism $W\cong S^n$. The simple reflection
$s_\a$ for any root $\a=e_i-e_j$ acts on the
weight lattice (spanned by the weights $e_1$,\ldots,$e_n$, which form an orthonormal
basis) by the reflection in the plane perpendicular to $e_i-e_j$ and hence permutes
the weights $e_1$,\ldots, $e_n$ by the transposition $(i~j)$.

\section{Schubert calculus for algebraic cobordism of flag varieties}
\label{s.main}
In this section, we assume that $G$ is an arbitrary connected split reductive group
unless we explicitly mention that $G=GL_n(k)$, and $X=G/B$ is the complete flag variety for $G$.
We now investigate the ring structure of $\Omega^*(X)$ in more geometric
terms.

\subsection{Schubert cycles and Bott-Samelson resolutions}
Recall that the flag variety $X$ is cellular with the following cellular
decomposition into {\em Bruhat cells}.
Let us fix a Borel subgroup $B$.
For each element $w\in W$ of the Weyl group of $G$, define the {\em Bruhat (or Schubert)
cell} $C_w$ as the $B$--orbit of
the the point $wB\in G/B=X$ (we identify the Weyl group with $N(T)/T$ for a maximal
torus $T$ of $G$ inside $B$).
The {\em Schubert cycle} $X_w$ is defined as the closure of $C_w$ in $X$.
The dimension of $X_w$ is equal to the length of $w$ \cite{BGG}. Recall that the
length of an element $w\in W$ is defined as the minimal number of factors in a
decomposition of $w$ into the product of simple reflections.
Recall also that for each $l$-tuple $I=(\a_1,\ldots,\a_l)$
of simple roots of $G$, one can define the {\em Bott-Samelson resolution}
$R_I$ (which has dimension $l$) together with the map $r_I:R_I\to X$. Bott-Samelson
resolutions are smooth. Consequently, for any $I$  the map
$r_I:R_I \to X$ represents an element in $\Omega^*(X)$ which we denote by $Z_I$.

Denote by $s_{\a}\in W$ the reflection corresponding to a root $\a$, and by
$s_I$ the product $s_{\a_1}\cdots s_{\a_l}$. If the decomposition
$s_I=s_{\a_1}\cdots s_{\a_l}$ defined by
$I$ is reduced (that is, $s_I$ can not be written as a product of less than
$l$ simple reflections, or equivalently, the length of $s_I$ is equal to $l$),
then the image $r_I(R_I)$ coincides with the Schubert cycle
$X_{s_I}$ (which we will also denote by $X_I$).
The dimension of $X_I$ in this case is also equal to $l$ and the map
$r_I:R_I\to X_{I}$ is a
birational isomorphism. In this case, the variety $R_I$ is
a resolution of singularities for the Schubert cycle $X_{I}$.

Bott-Samelson resolutions were introduced by Bott and Samelson in the case of compact Lie groups,
and by Demazure in the case of algebraic semisimple groups \cite{De}.  There are
several equivalent definitions, see e. g. \cite{BK,De,Ma}.
We will use the definition below (which follows easily from \cite[2.2]{BK}), since it
is most suited to our needs.
Namely, $R_I$ is defined by the following inductive procedure starting from
$R_{\emptyset}=pt=Spec(k)$ (in what follows we will rather denote $R_{\emptyset}$ by
$R_e$). For each $j$-tuple $J=(\a_1,\ldots,\a_j)$ with $j < l$, denote by
$J\cup \{j+1\}$ the $(j+1)$-tuple $(\a_1,\ldots,\a_j,\a_{j+1})$.
Define $R_{J\cup \{j+1\}}$ as  the fiber product $R_J\times_{G/P_{j+1}}G/B $, where $P_{j+1}$ is the minimal
parabolic subgroup corresponding to the root $\a_{j+1}$. Then
the map  $r_{J\cup\{j+1\}}:R_{J\cup\{j+1\}}\to X$ is defined as the projection to
the second factor.
In what follows, we will use that $R_J$ can be
embedded into $R_{J\cup \{j+1\}}$ by sending $x\in R_J$ to
$(x,r_J(x))\in R_J\times_{G/P_{j+1}}G/B$.

In particular, one-dimensional Bott-Samelson resolutions are isomorphic to the
corresponding Schubert cycles. It is easy to show that any two-dimensional
Bott-Samelson resolution $R_I$ for a reduced $I$ is also isomorphic
to the corresponding Schubert cycle. More generally, $R_I$ is isomorphic
to $X_I$ if and only if all simple roots in $I$ are pairwise distinct (in particular,
the length of $I$ should not exceed the rank of $G$). The simplest example where
$R_I$ and $X_I$ are not isomorphic for a reduced $I$ is $G=GL_3$ and
$I=(\g_1,\g_2,\g_1)$ (where $\g_1$, $\g_2$ are two simple roots of $GL_3$).

It is easy to show that $R_{J\cup\{j+1\}}$ is the projectivization
of the bundle $r_{J}^*\pi_{j+1}^*E$, where $E$ is the rank two vector bundle on
$G/P_{j+1}$ defined in the next subsection and $\pi_{j+1}:G/B\to G/P_{j+1}$ is the
natural projection. This is the definition used in \cite{BE}. In the topological
setting, the vector bundle $r_{J}^*\pi_{j+1}^*E$ splits into the sum of two line
bundles \cite{BE} but in in the algebro-geometric setting this is no longer true
(though $r_{J}^*\pi_{j+1}^*E$ still contains a line subbundle as follows from
the proof of Lemma \ref{l.Weyl}).

This definition of $R_I$ allows to describe easily (by repeated use of the projective bundle
formula) the ring structure of the cobordism ring $\Omega^*(R_I)$. It also implies that $R_I$
is cellular with $2^l$ cells labeled by all subindices $J\subset I$.

The cobordism classes $Z_I$ of Bott-Samelson resolutions generate
$\Omega^*(X)$ but do not form a basis.
The following proposition is an immediate corollary of Theorem \ref{t.cellular}.
An analogous statement for complex cobordism is proved in \cite[Proposition 1]{BE}
by using the Atiyah-Hirzebruch spectral sequence (as mentioned in Section 2).
\begin{prop}\label{p.BScell} As an $\L$-module, the algebraic cobordism ring
$\Omega^*(X)$ of the flag variety is freely generated by the Bott-Samelson
classes $Z_{I(w)}$ where $w\in W$ and $I(w)$ defines a reduced decomposition
for $w$ (we choose exactly one $I(w)$ for each $w$).
\end{prop}
There is no canonical choice for a decomposition $I(w)$ of a given element $w$ in the Weyl group.
From the geometric viewpoint it is more natural to consider all Bott-Samelson classes
at once (including those for non-reduced $I$) even though they are not linearly
independent over $\L$. So throughout the rest of the paper we will not put any restrictions on
the multiindex $I$.

\subsection{Schubert calculus} \label{ss.operators}
We will now describe the cobordism classes $Z_I$ as polynomials in the first Chern
classes of line bundles on $X$. This allows us to compute products
of Bott-Samelson resolutions and hence achieves the goal
of a Schubert calculus for algebraic cobordism.

We first define operators $A_i$  on $\Omega^*(X)$ following the approach
of the previous section (see Definition \ref{d.operator}).
These operators generalize the {\em divided difference operators} on the Chow ring
$CH^*(X)$ defined in
\cite{BGG,De,Che}  to algebraic cobordism.

We first define operators $A_i$ for $GL_n$ since in this case the Borel presentation
allows to make them more explicit.
We start with
the subgroup $B$ of upper triangular matrices and the diagonal torus,
which yields an isomorphism $W \cong S_n$. Under this isomorphism,
the reflection $s_{\a}$ with respect to a root $\a=e_i-e_j$ goes to the
transposition $(i~j)$ (see the end of Section 2).
For each positive root $\a$ of $G$, we define the operators $\sigma_\a$
and $\hat A_\a$ on the ring of formal power series $\L[[x_1,\ldots,x_n]]$ as follows:
$$(\sigma_\a f)(x_1,\ldots,x_n)=f(x_{s_\a(1)},\ldots,x_{s_\a(n)}),$$
$$\hat A_\a=(1+\sigma_\a)\frac1{F(x_{i+1},\chi(x_i))}.$$
It is easy to check that $\hat A_\a$ is
well-defined
on the whole ring $\L[[x_1,\ldots,x_n]]$ (see Section 5).
Note also that under the homomorphism
$\L[[x_1,\ldots,x_n]]\to \L[x_1,\ldots,x_n]/S\cong \Omega^*(X)$
the power series $F(x_{i+1},\chi(x_i))$ maps to $c_1(L(\g_i))$
(see the end of Section 2), so our definition for additive formal group law
reduces to the definition of divided difference operator on the polynomial ring
$\Z[x_1,\ldots,x_n]$ (see \cite[2.3.1]{Ma}).
Finally, we define the operator $A_\a:\Omega^*(X)\to \Omega^*(X)$ using the
Borel presentation by the formula
$$A_\a(f(x_1,\ldots,x_n))=\hat A_\a(f)(x_1,\ldots,x_n)$$
for each polynomial $f\in\L[x_1,\ldots,x_n]$. Again, by degree reasons the right hand
side is a polynomial. The operator $A_\a$ is well defined (that is, does not depend
on a choice of a polynomial $f$ representing a given class in $\L[x_1,\ldots,x_n]/S$)
since for any polynomial $h$ and any symmetric polynomial $g$ we have
$\hat A_\a(gh)=g\hat A_\a(h)$.

We now define $A_i=A_{\alpha_i}$ for an arbitrary reductive group $G$ and a simple root $\a_i$.
Denote by $P_i\subset G$ the minimal parabolic subgroup corresponding to the
root $\a_i$. Then $X=G/B$ is a projective line fibration over $G/P_i$. Indeed, consider the
projection $\pi_i:G/B\to G/P_i$. Take the line bundle $L(\rho)$ on $G/B$ corresponding to the
weight $\rho$, where $\rho$ is the half-sum of all positive roots or equivalently the sum of all
fundamental weights of $G$ (the weight $\rho$ is uniquely characterized by the property that
$(\rho,\a)=1$ for all simple roots $\a$). Then it is easy to check that the vector bundle
$E:={\pi_i}_*L(\rho)$ on $G/P_i$ has rank two and $G/B=\P(E)$.
Note that tensoring $E$ with any line bundle $L$ on $G/P_i$ does not change $\P(E)=\P(E\otimes L)$ so
the property $\P(E)=X$ does not uniquely define the bundle $E$. However,  the choice
$E={\pi_i}_*L(\rho)$ (suggested to us by Michel Brion) is the only uniform choice for all $i$,
since $L(\rho)$ is the only line bundle on $X$ with the property $\P({\pi_i}_*L(\rho))=X$ for all $i$.
We now use Definition \ref{d.operator} to define an $\Omega^*(G/P_i)$-linear operator
$A_i:=A_{\pi_i}$
on $\Omega^*(X)$. For $G=GL_n$, this definition coincides with the one given above.
This is easy to show using that $G/P_i$ for $\a_i=\g_i$ is the partial flag variety
whose points are flags $F=\{\{0\}=F^0\subset\ldots\subset
F^{i-1}\subset F^{i+1}\subset \ldots\subset F^n=k^n\}$.

Let $I=(\a_1,\ldots,\a_l)$ be an $l$-tuple of
simple roots of $G$. Define the element $\mathcal R_I$ in
$\Omega^*(X)$ by the formula

$$\R_{I}:=A_{l}\ldots A_{1}Z_e.$$

In the case $G=GL_n$, we can also regard $\R_I$ as a polynomial in
$\L[x_1,\ldots,x_n]/S$.
Similar to \cite[Theorem 3.15]{BGG} or \cite[page 807]{BE},
one may describe $Z_e$ for general $G$ using the formula
$$Z_e=\R_e:=\frac1{|W|}\prod_{\a\in R^+}c_1(L(\a)),$$
where $R^+$ denotes the set of positive roots of $G$
(recall that $|R^+|=\dim~X=:d$). As in the Chow ring case,
there is also the formula
$$Z_e=\frac{1}{d!} L(\rho)^d.$$
Both formulas immediately follow from the analogous formulas for
the Chow ring \cite[Theorem 3.15, Corollary 3.16]{BGG} since
$\Omega^d(X)\simeq CH^d(X)$ (as follows from
\cite[Theorem 1.2.19, Remark 4.5.6]{LM}).

Note that for $GL_n$ the formula for $\R_e$ reduces to $\R_e=\Delta_n$
since $c_1(L(e_i-e_j))=x_j-x_i+$higher order terms, and hence the equality $Z_e=\R_e$
follows from
Remark \ref{r.Borel}. In particular, by the same remark $\R_e$ modulo $S$ has a
denominator-free expression $x_n^{n-1}x_{n-1}^{n-2}\cdots x_2$.

We now prove an algebro-geometric version of \cite[Corollary 1, Proposition 3]{BE} using
our algebraic operators $A_i$.
\begin{thm}\label{main} The cobordism class $Z_I=[r_I:R_I\to X]$ of the
Bott-Samelson resolution $R_I$ is equal to $\R_I$.
\end{thm}

\begin{proof} The essential part of the proof is the formula for the push-forward as
stated in Corollary \ref{c.Gysin}.
Once this formula is established it is not hard to show that
$A_iZ_I=Z_{I\cup\{i\}}$ for all $I$ by exactly the same methods as in the Chow ring
case \cite{Ma} and in the complex cobordism case \cite{BE}. Namely,
we have the following
cartesian square
$$\begin{CD}
G/B\times_{G/P_i}G/B @>p_2>> G/B\\
@V p_1 VV @VV \pi_i V\\
G/B @>\pi_i>> G/P_i.
\end{CD}.$$
E.g., if $G=GL_n$ we get exactly the diagram
of \cite[proof of Lemma 3.6.20]{Ma}.
Using this commutative diagram and the definition
of Bott-Samelson resolutions it is easy to show that
${\pi_i}^*{\pi_i}_*Z_I=Z_{I\cup\{i\}}$ \cite[proof of Proposition 2.1]{BE}.
We now apply Corollary \ref{c.Gysin} and get that
$A_i={\pi_i}^*{\pi_i}_*$. It follows by induction on the length of $I$
that $Z_I=A_l\ldots A_1 Z_e$.
\end{proof}

\begin{remark} \em Note that if we apply the base change formula
\cite[Definition 1.1.2 (A2)]{LM}
to the cartesian diagram from the proof
of Theorem \ref{main}, we get ${p_1}_*p_2^*=\pi_i^*{\pi_i}_*$, where the
right hand side is precisely the
definition of the ``geometric'' operator denoted $A_i$ in \cite{BE}, while the
left hand side is the operator denoted $\delta_i$ in \cite[proof of Theorem 3.6.18]{Ma}.
Hence Manivel and Bressler--Evens consider the same operators.
\end{remark}

We now compute the action of the operator $A_i$ on polynomials in the first
Chern classes (this computation will be used in Sections 4 and 5).
Consider the operator
$\sigma_i:=\sigma_{\pi_i}$ again defined as in Definition
\ref{d.operator}. Note that $\sigma_i$ corresponds to the simple reflection
$s_i:=s_{\a_i}$ in the following sense.

\begin{lemma} \label{l.Weyl} For any line bundle $L(\l)$ on $X$, we have
$$\sigma_{i}(c_1(L(\l)))=c_1(L(s_{i}\l)).$$
\end{lemma}
\begin{proof}
Since $X=\P(E)$ (recall that $E={\pi_i}_*L(\rho)$), the bundle $\pi_i^*E$
on $X$ admits the usual short exact sequence
$$0\to \tau_E\to \pi_i^*E\to \Oc_E(1)\to 0,$$
where $\tau_E$ is the tautological line bundle on $X$ (that is, the fiber of $\tau_E$ at the
point $x\in X=\P(E)$ is the line in $E$ represented by $x$). Note that in our case $\P(E)=\P(E^{dual})$
since $E$ is of rank two (thus hyperplanes in $E$ are the same as lines in $E$).
It is easy to show that there is an isomorphism of line bundles $$\tau_E^{-1}\otimes\Oc_E(1)=L(\a_i).$$
(Moreover, one can show that $\tau_E=L(\rho-\a_i)$ and $\Oc_E(1)=L(\rho)$.)
Indeed, $\tau_E^{-1}\otimes\Oc_E(1)\simeq Hom(\tau_E,\Oc_E(1))$ can be thought of as the
bundle of tangents along the fibers of $\pi_i$. The latter is the line bundle associated with the
$B$--module ${\mathfrak p_i}/{\mathfrak b}$, which has weight $-\a_i$ (see \cite[Remark 1.4.2]{Brion}
for an alternative definition of the line bundles $L(\l)$ in terms of the one-dimensional $B$--modules).
Here ${\mathfrak p_i}$ and ${\mathfrak b}$ denote the Lie algebras of $P_i$ and $B$, respectively.

By definition, $\sigma_i$ switches $c_1(\tau_E)$ and $c_1(\Oc_E(1))$.
Hence, $\sigma_i(c_1(L(\a_i)))=c_1(L(-\a_i))$. Since the Picard group of $G/P_i$
can be identified with with the sublattice $\{\l| (\l,\a_i)=0\}$ of the weight lattice
of $G$ (this follows from \cite[remark after Proposition 1.3.6]{Brion}
combined with \cite[Proposition 1.4.3]{Brion}) we also have
$\sigma_i(c_1(L(\l)))=c_1(L(\l))$ for all $\l$ perpendicular to $\a_i$. These two identities
imply the statement of the lemma.
\end{proof}

This lemma allows us to describe explicitly the action of $\sigma_i$
and hence of $A_i$ on any polynomial in the first Chern classes. Indeed, since
for any weight $\l$ we have $s_i\l=\l+k\a_i$ for some integer $k$, we can compute
$c_1(L(\sigma_i\l))=c_1(L(\l)\otimes L(\a_i)^k)$ as a power series in $c_1(L(\l))$ and
$c_1(L(\a_i))$ using the formal group law.
This will be used in the proof of Proposition \ref{p.Chevalley} below and in
Subsection 5.1.

\section{Chevalley-Pieri formulas}\label{s.CP}

A key ingredient for the classical Schubert calculus is the
Chevalley-Pieri formula for the product of the Schubert cycle
with the first Chern class of the line bundle on $X$, see e. g.
\cite[Proposition 4.1]{BGG} and \cite[Proposition 4.2]{De}.
We now establish analogous formulas for the products of
$Z_I$ and $\R_I$ with $c_1(L(\l))$ (without using that $Z_I=\R_I$).
At the end of this section, we explain why
in the case of algebraic cobordism this alone is not enough
to show that $Z_I=\R_I$, hence justifying our
different approach of the previous two sections.

\medskip

By $L(D)$ denote the line bundle corresponding
to the divisor $D$.
For each $l$-tuple $I$ as above, denote by $I^j$ the
$(l-1)$--tuple $(\a_1,\ldots,\hat \a_j,\ldots,\a_l)$.
For each root $\a$, define the linear function $(\cdot,\a)$ (that is, the
{\em coroot})
on the weight lattice of $G$ by the property $s_\a\l=\l-(\l,\a)\a$ for all weights $\l$.
(The pairing  $(a,b)$ is often denoted by $\langle a,b^\vee\rangle$
or by $\langle a,b\rangle$.)
Note that by
definition $(\l,\a)=(w\l,w\a)$ for all elements $w$ of the Weyl group.

\begin{prop}{\bf \em Geometric Chevalley-Pieri formula}\label{p.bundle}

{\em (1) (for Bott-Samelson resolutions)}
In the Picard group of $R_I$ we have
$$r_I^*L(\l)=\otimes_{j=1}^{l}L(R_{I^j})^{(\l,\b_j)}$$
where $\b_j=s_{l}\cdots s_{j+1}\a_j$.

{\em (2) (for Schubert cycles)}\cite[Proposition 4.1]{BGG},
\cite[Proposition 4.4]{De}, \cite{Che}
In the Chow ring of $X$ we have
$$c_1(L(\l))X_I=\sum_{j}(\l,\b_j)X_{I^j}$$
where the sum is taken over $j\in \{1,\ldots,l\}$ for which the decomposition
defined by $I^j$ is reduced.
\end{prop}
The first part of this proposition was proved in \cite[Proposition 4]{BE} in the
topological setting (for flag varieties of compact Lie groups). It is not hard to check that the proof
carries over to algebro-geometric setting. We instead provide a shorter proof along the same lines.
Our proof is based on the following lemma.
\begin{lemma}\cite[Proposition 2.1]{De}\label{l.bundle}
Let $p:R_I\to R_{I^l}$ be the natural projection
(coming from the fact that we defined $R_I$ as a projective bundle over $R_{I^l}$).
Then we have an isomorphism
$$r_I^*L(\l) \cong p^*r_{I^l}^*L(s_l\l)\otimes L(R_{I^l})^{(\l,\a_l)}$$
of line bundles on $R_I$.
\end{lemma}

Proposition \ref{p.bundle}(1) now follows from Lemma \ref{l.bundle} by induction on $l$.
The base $l=1$,
that is $r_1^*L(\l)=\Oc_{\P^1}(1)^{(\l,\a_1)}$, follows from the fact that
$r_1:R_1\to X$ maps $R_1$ isomorphically to $P_{1}/B\cong\P^1$, which can be regarded
as the flag variety for $SL_2$. Then the weight $\l$ restricted to
$SL_2$ is equal to $(\l,\a_1)$ times the highest weight of the tautological
representation of $SL_2$, which corresponds to the line bundle $\Oc_{\P^1}(1)$ on $\P^1$.
 To prove the induction step plug in the induction
hypothesis for $r_{I^l}^*L(s_l\l)=\otimes_{j=1}^{l-1}L(R_{I^{j,l}})^{(s_l\l,s_{l-1}\cdots s_{j+1}\a_j)}$
into the lemma and use that $(s_l\l,s_{l-1}\cdots s_{j+1}\a_j)=(\l,\b_j)$ (since $s_l^2=e$) and
$p^*R_{I^{j,l}}=R_{I^j}$.

Proposition \ref{p.bundle}(1) was used in \cite{BE} to establish
an algorithm for computing $c_1(L(\l))Z_I$ in
$\Omega^*(X)$ \cite{BE}. We now briefly recall this algorithm.
By the projection formula we have
$$c_1(L(\l))Z_I=(r_I)_*(c_1(r_I^*L(\l))).$$
Note that the usual projection formula
with respect to smooth projective morphisms $f:X \to Y$
holds for algebraic cobordism as well. This follows from
the definition of products via pull-backs along the
diagonal and the base change axiom $(A2)$ of \cite{LM}
applied to the cartesian square obtained from
$Y \stackrel{diag}{\to} Y \times Y \stackrel{p \times id}{\leftarrow}
X \times Y$.

One can now use Proposition \ref{p.bundle}(1) and the formal group law
to compute $c_1(r_I^*L(\l))$ in terms of the Bott-Samelson classes in $\Omega^*(R_I)$
by an iterative procedure
(since the multiplicative structure of $\Omega^*(R_I)$ can be determined by the
projective bundle formula and the Chern classes arising this way again have form
$c_1(L(\l))$ for some $\l$). After $c_1(r_I^*L(\l))$ is written as
$\sum_{J\subset I}a_J[R_J]$ for some $a_J\in\L$ it is easy
to find $(r_I)_*(c_1(r_I^*L(\l)))$ since $(r_I)_*[R_J]=Z_J$.

However, this procedure is rather lengthy, and we will not use it. Instead, we will prove
a more explicit formula for $c_1(L(\l))Z_I$ (see formula 5.1 below) using our
algebraic Chevalley-Pieri formula together with Theorem \ref{main}.
\begin{prop}\label{p.Chevalley} {\bf \em  Algebraic Chevalley-Pieri formula:}

{\em (1) (cobordism version)}
Let
$A_1=A_{\a_1}$,\ldots, $A_l=A_{\a_l}$ be the operators on $\Omega^*(X)$ corresponding
to $\a_1$,\ldots, $\a_l$. Then we have
$$c_1(L(\l))A_1\ldots A_l\R_e=\sum_{j=1}^{l}A_1\ldots A_{j-1}
\frac{c_1(L(\l_j))-c_1(L(s_j\l_j))}
{c_1(L(\a_j))}A_{j+1}\ldots A_l\R_e$$
in $\Omega^*(X)$, where $\l_j=s_{j-1}\cdots s_1\l$ and $s_j=s_{\a_j}$ is the
reflection corresponding to the root $\a_j$.

{\em (2) (Chow ring version)} \cite[Corollary 3.7]{BGG} Let
$A_1=A_{\a_1}$,\ldots, $A_l=A_{\a_l}$ be the operators on $CH^*(X)$
corresponding to $\a_1$,
\ldots, $\a_l$. Then
$$c_1(L(\l))A_1\ldots A_l\R_e=\sum_{j=1}^{l}(\l,s_{1}\cdots s_{j-1}\a_j)A_1\ldots \hat A_j\ldots A_l\R_e$$
in $CH^*(X)$.
\end{prop}
\begin{proof}
First, note that $\frac{c_1(L(\l_j))-c_1(L(s_j\l_j))}
{c_1(L(\a_j))}$ is a well-defined element in $\Omega^*(X)$ because $s_j\l=\l-(\l,\a_j)\a_j$
(and hence $L(\l)=L(s_j\l)\otimes L(\a_j)^{(\l,\a_j)}$) and the formal group law  expansion for
$c_1(L_1\otimes L_2^k)-c_1(L_1)$
is divisible by $c_1(L_2)$ for any integer $k$ \cite[(2.5.1)]{LM}.
Next we show that
$$c_1(L(\l))A_1-A_1c_1(L(s_1\l))=\frac{c_1(L(\l))-c_1(L(s_1\l))}{c_1(L(\a_1))},$$
where both sides are regarded as operators on
$\Omega^*(X)$.
Indeed, by definition $A_1=(1+\sigma_1)\frac{1}{c_1(L(\a_1))}$ and
$c_1(L(\l))\sigma_1=\sigma_1c_1(L(s_1\l))$ by Lemma \ref{l.Weyl}.

Hence, we can write
$$c_1(L(\l))A_1\ldots A_l\R_e=\frac{c_1(L(\l))-
c_1(L(s_1\l))}{c_1(L(\a_1))}A_2\ldots A_l\R_e+A_1c_1(L(s_1\l))A_2\ldots A_l\R_e,
$$
and then continue moving $c_1(L(s_1\l))$ to the right until we are left with with
the term $A_1\ldots A_lc_1(L(s_l\ldots s_1\l))\R_e$. This term is equal to zero
since $c_1(L(s_l\ldots s_1\l))\R_e$ is the product of more than $\dim X$ first
Chern classes, and hence its degree is greater than $\dim X$.
The Chow ring case follows immediately from the cobordism case since
$$\frac{c_1(L(\l))-c_1(L(s_j\l))}{c_1(L(\a_j))}=(\l,\a_j)$$
in the Chow ring. The last identity holds because the formal group law  for the Chow ring is additive,
and hence $c_1(L(\l))-c_1(L(s_j\l))=
(\l,\a_j)c_1(L(\a_j))$.
\end{proof}
The second part of this proposition was proved in \cite{BGG} by more involved
calculations. A calculation similar to ours was used in \cite{RP} to deduce a
combinatorial Chevalley-Pieri formula for $K$-theory. It would be interesting to
find an analogous combinatorial interpretation of our Chevalley-Pieri formula
in the cobordism case.

Note that in the case of Chow groups, the algebraic Chevalley-Pieri formula for
$A_l\ldots A_1\R_e$ is exactly the same as
the geometric one for the Schubert cycle $X_I$. Together with the Borel presentation
this easily implies  that the polynomial
$A_l\ldots A_1\R_e$ represents the Schubert cycle $X_I$ whenever $I$
defines a reduced decomposition \cite{BGG}.
Indeed, we can proceed by the induction on $l$.  Algebraic and geometric Chevalley-Pieri formulas allow to compute the
intersection indices of $A_l\ldots A_1\R_e$ and of $X_I$, respectively, with the product of $k$ first Chern classes,
and the result is the same in both cases by the induction hypothesis (for all $k>0)$.
By the Borel presentation we know that the  products of
first Chern classes span $\oplus_{i=1}^d CH^i(X)$. Hence, by the non-degeneracy of the
intersection form on $CH^*(X)$ (that is, by Poincar\'e duality) we have that $A_l\ldots A_1\R_e-X_I$
must lie in $CH^{d}(X)=\Z[pt]$ (that is, in the orthogonal complement to $\oplus_{i=1}^d CH^i(X)$).
This is only possible if $A_l\ldots A_1\R_e-X_I=0$ (unless $l=0$, which is the induction base).
Note that the only geometric input in this proof is the geometric Chevalley-Pieri formula.

In the cobordism case, it is not immediately clear why geometric and algebraic
Chevalley-Pieri formulas are the same  (though, of course, it follows from
Theorem \ref{main}). But even without using that $\R_I=Z_I$ it might be possible to
show that both formulas have the same structure
coefficients, that is, if $c_1(L(\l))Z_I=\sum_{J\subset I}a_JZ_J$ then necessarily
$c_1(L(\l))\R_I=\sum_{J\subset I}a_J\R_J$ with the same coefficients $a_J\in\L$. However, this does not
lead to the proof of $\R_I=Z_I$ as in the case of the Chow ring. The reason
is that even though there is an analog of Poincar\'e duality for the cobordism rings of
cellular varieties, this only yields an equality $\R_I=Z_I$ up to a multiple of $[pt]$
(as in Lemma \ref{l.poincare} below in the case of $GL_n/B$),
and this is not enough to carry out the desired induction argument.
For the Chow ring, Poincar\'e duality also yields only an equality up to the class of a point,
but unless $I=\emptyset$, the difference $\R_I-Z_I$ (where now $Z_I$ means the Schubert cycle and not the Bott-Samelson class)
can not be a non-zero multiple of $[pt]$ because the  coefficient ring $CH^*([pt])=CH^*(k)\cong \Z$ is
concentrated in degree zero, hence has no nonzero elements in the corresponding degree $l-d$. However,
for algebraic cobordism the coefficient ring $\Omega^*(k) \cong \L$ does contain plenty of elements
of negative degree, so one can not deduce $\R_I=Z_I$.

\begin{lemma} \label{l.poincare}
Let $X=GL_n/B$ and let $c\in\Omega^*(X)\cong
\L[x_1,\ldots,x_n]/S$ be a homogeneous element of degree
$l$ such that the product of $c$ with any non-constant monomial
in $x_1$,\ldots,$x_n$ is zero.
Then $c$ belongs to $\L^{l-d}[pt]$, where $d=n(n-1)/2$.
\end{lemma}

\begin{proof} We will use that the ideal $S$ contains all homogeneous polynomials
of degree greater than $d$ with integer coefficients \cite[Corollary 2.5.6]{Ma}.
Let $\hat{c}$ be a homogenous element in $\L[x_1,...,x_n]$ that
represents $c$. Recall that $\L$ is isomorphic to the graded polynomial
ring $\Z[a_1,a_2,\ldots]$
in countably many variables, where $a_i$ has degree $-i$.
Therefore $\hat{c}$ has a unique decomposition as a sum
of integral polynomials with coefficients being
monomials in the $a_i$, that is
$$\hat{c}=c_0+ a_1c_{1}+ a_2 c_2 + a_1^2 c_{1,1} + a_3 c_3
+ a_1 a_2 c_{1,2} + a_1^3 c_{1,1,1} + \ldots$$
where $c_{i_1,...,i_s}$ is a polynomial of degree $l-\sum \deg(a_{i_j})$
with integer coefficients. Note that we might
choose a $\hat{c}$ such that the sum is
finite since $c_{i_1,...,i_s}$ vanishes modulo $S$ if
$l-\sum \deg(a_{i_j}) > d$.
Now we multiply $\hat{c}$ with an arbitrary monomial $m_{d-l}$
in the $x_i$ of degree $d-l$. Since $m_{d-l}c=0$  it follows that $m_{d-l}c_0$ is zero modulo $S$.
By algebraic Poincar\'e duality \cite[Proposition 2.5.7]{Ma}
it follows that $c_0=0$
modulo $S$. Next, we multiply with monomials of degree
$d-l-1$ to deduce that $c_1$ equals zero modulo $S$,
and then deduce inductively that all the $c_{i_1,...,i_s}$ of degree
strictly less than $d$ are zero. It remains to note that each
$c_{i_1,...,i_s}$ of degree $d$ is equal to an integer multiple of $[pt]$
(since all homogeneous polynomials of degree $d$ with integer coefficients are equal
to a multiple of $\R_e$ modulo the ideal $S$ \cite[2.5.2]{Ma})
Hence $\hat{c}=a[pt]$ for some $a\in\L$, which must have degree $l-d$ by homogeneity
of $\hat{c}$.
\end{proof}

\section{Computations and examples}

Until now, we used the formal group law  of algebraic cobordism
(i.e., the universal one) as little as possible in order to make our
presentation simpler. In this section, we make the results of the previous
section more explicit using this formal group law.
In particular, we give an explicit formula for the products of a Bott-Samelson
resolution with the first Chern class of a line bundle in terms of other Bott-Samelson
resolutions (see formula 5.1 below). Using this formula, we give an algorithm
for computing the product of two Bott-Samelson resolutions.

First, we show that the operator $A$ from Section 2 and the operator $\hat A_{\a}$
from Section 3 are well-defined.

We use notation of Subsection 2.1, so $F(u,v)$ is the universal formal group law
and $\chi(u)$ is the inverse for the universal formal group law  defined by the identity
$F(u,\chi(u))=0$. To show that the operator $A=(1+\sigma)\frac{1}{F(y_1,\chi(y_2))}$
is well defined on $\Omega^*(X)[[y_1,y_2]]$  it is enough to show that
$A(m)$ is a formal power series
for any monomial $m=y_1^{k_1}y_2^{k_2}$.
We compute
$A(y_1^{k_1}y_2^{k_2})$ using that $y_1=F(x,y_2)=y_2+\chi(x)p(x,y_2)$ and
$y_2=F(\chi(x),y_1)=y_1+\chi(x)p(\chi(x),y_1)$ where $x=F(y_1,\chi(y_2))$ and
$p(u,v)=\frac{F(u,v)-u}{v}$ is a well-defined power series (since $F(u,v)-u$
contains only terms $u^iv^j$ for $j\ge1$).
 We get
$$A(y_1^{k_1}y_2^{k_2})=(1+\sigma)\frac{y_1^{k_1}y_2^{k_2}}{x}=
\frac{y_1^{k_1}y_2^{k_2}}{x}+\frac{y_1^{k_2}y_2^{k_1}}{\chi(x)}
=\frac{(y_2+\chi(x)p(x,y_2))^{k_1}(y_1+\chi(x)p(\chi(x),y_1))^{k_2}}{x}+
\frac{y_2^{k_1}y_1^{k_2}}{\chi(x)}=$$
$$=y_2^{k_1}y_1^{k_2}q(x,\chi(x))+
\frac{\mbox{ terms divisible by } x \mbox{ or by } \chi(x)}{x}.$$
The second term in the last expression is a power series since the formal group law
expansion for $\chi(x)$ is divisible
by $x$  \cite[(2.5.1)]{LM}.

A similar argument shows that the operator $\hat A_{\a}$
from Section 3 is indeed well-defined on the whole ring
$\L[[x_1,\ldots,x_n]]$ for any root $\a$. Indeed, by relabeling
$x_1$,\ldots, $x_n$ we can assume that $\a=e_1-e_2$. Then for any monomial
$m=x_1^{k_1}x_2^{k_2}\ldots x_n^{k_n}$ we have
$$\hat A_{\a}(m)=x_3^{k_3}\ldots x_n^{k_n}\hat A_{\a}(x_1^{k_1}x_2^{k_2}).$$
Then exactly the same argument as the one above for $A$ shows that
$\hat A_{\a}(x_1^{k_1}x_2^{k_2})$ is a power
series in $x_1$ and $x_2$.

\subsection{Algorithm for computing the products of Bott-Samelson resolutions}
We now produce an explicit algorithm for computing the product of the
Bott-Samelson classes
$Z_I$  in terms of other Bott-Samelson classes,
where $I=(\alpha_1,...,\alpha_l)$.
The key ingredient is our algebraic Chevalley-Pieri
formula (Proposition \ref{p.Chevalley}) which can be reformulated as follows
$$c_1(L(\l))A_1\ldots A_lZ_e=\sum_{j=1}^{l}A_1\ldots A_{j-1}
A_j^*(c_1(L(\l_j)))A_{j+1}\ldots A_lZ_e,$$
where $\l_j=s_{j-1}\cdots s_1\l$ (in other words, $c_1(L(\l_j))=[\sigma_{j-1}\ldots\sigma_1](c_1(L(\l)))$)
and the operator $A_j^*$ is defined as follows
$$A_j^*=A_{\a_j}^*=\frac{1}{c_1(L(\a_j))}(1-\sigma_{\a_j}).$$
We can compute $A_j^*$ on any polynomial in the first Chern classes by the same methods
as $A_j$ (see the end of Section 3). Note that for the Chow ring $A_j=A_j^*$
(this follows from Lemma 3.4
 and the fact that $s_j\a_j=-\a_j$ and $c_1(L(\a_j))=-c_1(L(-\a_j))$ for the additive formal group law),
but for the algebraic cobordism ring this is no longer true.

More generally, for any polynomial $f=f(c_1(L(\mu_1)),\ldots, c_1(L(\mu_k)))$ in the first Chern classes
of some line bundles on $X$, we can compute its product with $A_1\ldots A_lZ_e$ by exactly the same argument as in the proof of Proposition \ref{p.Chevalley}:
$$f\cdot A_1\ldots A_lZ_e=\sum_{j=1}^{l}A_1\ldots A_{j-1}
[A_j^*\sigma_{j-1}\ldots\sigma_1](f)A_{j+1}\ldots A_lZ_e+
A_1\ldots A_l[\sigma_l\ldots\sigma_1](f)Z_e \ \ (5.0)$$
Note that the last term on the right hand side is equal to the constant term of the polynomial
$[\sigma_l\ldots\sigma_1](f)$ (which is of course the same as the constant term of $f$) times
$A_1\ldots A_lZ_e$.
In particular, for $f=c_1(L(\l))$ this term vanishes modulo $S$.
Here and below, by the ``constant term'' of a polynomial in
$\L[x_1,\ldots,x_n]$ we mean the term of polynomial degree zero (the total
degree of such a constant term might be negative since the Lazard
ring $\L$ contains elements of negative degree).
Note that all elements of
$\L\subset \L[x_1,\ldots,x_n]$ are invariant under the operators $\sigma_i$,
and hence commute with the operators $A_i$. For an arbitrary reductive group, the constant term of
an element $f\in\Omega^*(X)$
is defined as the product of $f$ with the class of a point.

It is now easy to show by induction on $l$ that
$$fA_1\ldots A_lZ_e=\sum_{J\subset I}a_J(f)[\prod_{i\in{I\setminus J}}A_i]Z_e,$$
where $a_J(f)$ for the $k$-subtuple $J=(\a_{j_1},\dots,\a_{j_k})$ of $I$ is
the constant term in the expansion for
$[\sigma_l\ldots\sigma_{j_k+1}A_{j_k}^*\sigma_{j_k-1}\ldots
\sigma_{j_1+1}A_{j_1}^*\sigma_{j_1-1}\ldots\sigma_1]f$,
which is invariant under $\sigma_i$ (for all $i$) and hence equal to
$[A_{j_k}^*\sigma_{j_k-1}\ldots
\sigma_{j_1+1}A_{j_1}^*\sigma_{j_1-1}\ldots\sigma_1]f$.
Indeed, we first use formula (5.0) above and then apply the induction hypothesis to all terms in the
right hand side except for
the last term, which already has form $a_J(f)[\prod_{i\in{I\setminus J}}A_i]Z_e$
for $J=\emptyset$. We get
$$A_1\ldots A_{j-1}
[A_j^*\sigma_{j-1}\ldots\sigma_1](f)A_{j+1}\ldots A_lZ_e=$$
$$=A_1\ldots A_{j-1}
\sum_{J\subset I\setminus\{1,\ldots,j\}}a_J([A_j^*\sigma_{j-1}\ldots\sigma_1](f))
[\prod_{i\in{I\setminus (J\cup \{1,\ldots,j\})}}A_i]Z_e=$$
$$=\sum_{J'\subset I}a_{J'}(f)[\prod_{i\in{I\setminus {J'}}}A_i]Z_e,$$
where the last summation goes over all subsets $J'$ of $I$ that do contain $j$ but do
not contain $1$,\ldots,$j-1$. Plugging this back into formula (5.0) we get the desired formula.
Combining this with Theorem \ref{main}, we get the following formula
in $\Omega^*(X)$ for the product
of the Bott-Samelson class $Z_I$ with the first Chern class $c_1(L(\l))$
in terms of other Bott-Samelson classes
$$c_1(L(\l))Z_I=\sum_{J\subset I}b_J(\l)Z_{I\setminus J}, \eqno (5.1)$$
where $b_J(\l)$ is the constant term in the expansion for
$$[A_{j_1}^*\sigma_{j_1+1}\ldots
\sigma_{j_k-1}A_{j_k}^*\sigma_{j_k+1}\ldots\sigma_l](c_1(L(\l))).$$
We changed the order of the $\sigma_i$ when passing from $a_J$
to $b_J$ since $Z_I=A_l\ldots A_1Z_e$.
Note that for $J=\emptyset$ we have $b_J=0$, and for $J=(\a_j)$ we have $b_J=(\l,\b_j)$
since the constant term in
$A_j^*(c_1(L(s_{j+1}\ldots s_{l}\l)))=
\frac{c_1(L(s_{j+1}\ldots s_{l}\l))-c_1(L(s_js_{j+1}\ldots s_{l}\l))}{c_1(L(\a_j))}$
is equal to $(s_{j+1}\ldots s_{l}\l,\a_j)$
(see the proof of Proposition \ref{p.Chevalley}, and Proposition \ref{p.bundle} for the definition of $\beta_i$), which is equal to $(\l,\b_j)$.
So the lowest order terms (with respect to the polynomial grading) of this formula give an analogous formula for the Chow ring as expected.
\medskip

We now have assembled all necessary tools for actually performing the
desired Schubert calculus. Namely, to compute the product $Z_IZ_J$ we apply
the following procedure (which is formally similar
to the one for the Chow ring).
We replace $Z_J$ with the respective polynomial $\R_J$ in the first Chern classes
(using Theorem \ref{main} together with the formula for $Z_e$) and then compute
the product of $Z_I$ with each monomial in $\R_J$ using repeatedly formula (5.1).
Note that formula (5.1) allows us to make this algorithm more explicit than the one
given in \cite{BE} (see an example below).

\medskip

The naive approach to represent both $Z_I$ and $Z_J$ as fractions of polynomials
in first Chern classes and then computing their product is less useful.
In particular translating the product of the fractions back into
a linear combination of Bott-Samelson classes will be very hard,
if possible at all.

\subsection{Examples}

We now compute the Bott-Samelson classes $Z_I$ in terms
of the Chern classes $x_i$ for the example $X=SL_3/B$ where $B$ is the subgroup of upper-triangular
matrices.
We then compute certain products of Bott-Samelson classes
in two ways, by hand and then using the algorithm above together with formula (5.1).
Note that only the second approach generalizes to higher dimensions.

In $SL_3$, there are two simple roots $\g_1$ and $\g_2$.
In $X$, there are six Schubert cycles $X_{e}=pt$, $X_1$, $X_2$, $X_{12}$, $X_{21}$ and $X_{121}=X$
(here $12$ is a short hand notation for $(\g_1,\g_2)$, etc.).
Each $X_I$ except for
$X_{121}$ coincides with its Bott-Samelson resolution $R_I$. Note that
in general $R_I$ and $X_I$ do not coincide even when $X_I$ is smooth.
(By the way, for $G=GL_n$ the first non-smooth Schubert cycles
show up for $n=4$.)
\paragraph{\bf Computing $Z_I$ as a polynomial in the first Chern classes.}
We want to express $Z_I$ as a polynomial in $x_1$, $x_2$, $x_3$
using the formulas
$$Z_{s_{i_1}\ldots s_{i_l}}=A_{i_l}\ldots A_{i_1}\R_e;
\quad \R_e=\frac16c_1(L(\g_1))c_1(L(\g_2))c_1(L(\g_1+\g_2)).$$
Note that in computations involving  the operators $A_\a$ it is more convenient
not to replace $c_1(L(\a))$ with its expression in terms of $x_i$ until the very end.

Let us for instance compute $R_1$ as a polynomial in $x_1$, $x_2$, $x_3$ modulo the ideal $S$
generated by the symmetric polynomials of positive degree:
 $$\R_1=A_1\R_e=\frac16(1+s_1)c_1(L(\g_2))c_1(L(\g_1+\g_2))=$$
$$\frac13c_1(L(\g_2))c_1(L(\g_1+\g_2))=\frac13
F(\chi(x_2),x_3)F(\chi(x_1),x_3).$$
We have $\chi(u)=-u+a_{11}u^2-a_{11}^2u^3$ and $F(u,v)=u+v+a_{11}uv+
a_{12}u^2v+a_{21}uv^2$, where $a_{11}=-[\P^1]$ and $a_{12}=a_{21}=[\P^1]^2-[\P^2]$ \cite[2.5]{LM}.
Thus
$$\frac13
F(\chi(x_2),x_3)F(\chi(x_1),x_3)=\frac13F(-x_2+a_{11}x_2^2-a_{11}x_2^3,x_3)
F(-x_1+a_{11}x_1^2-a_{11}x_1^3,x_3)=$$
$$=\frac13(-x_2+x_3+a_{11}x_2^2-a_{11}x_2x_3)(-x_1+x_3+a_{11}x_1^2-a_{11}x_1x_3)=x_3^2,$$
since $(x_3-x_2)(x_3-x_1)=3x_3^2\mod S$, and $(x_2+x_1)(x_2-x_3)(x_1-x_3)=-3x_3^3=0 \mod S$.
So the answer agrees with the one we got in Remark \ref{r.Borel}.

Here are the polynomials for the other Bott-Samelson resolutions:
$$\R_{212}=1+a_{12}x_1^2;\quad \R_{121}=1+a_{12}x_1x_2$$
$$\R_{12}=-x_1-[\P^1]x_1^2 \quad \R_{21}=x_3=-x_1-x_2$$
$$\R_1=x_3^2=x_1x_2 \quad \R_2=x_1^2$$
$$\R_e=-x_1^2x_2.$$
Note that the Bott-Samelson resolutions $R_I$ in this list
coincide with the Schubert cycles they resolve if $I$ has length $\le 2$.
The corresponding polynomials $\R_I$ are the classical Schubert polynomials
(see e.g. \cite{Ma} and keep in mind that his $x_i$ is equal to our $-x_i$)
except for the polynomial $\R_{12}$.

In general,  polynomials $\R_I$ can be computed by induction on the length of $I$. E.g.
to compute $\R_{212}$ we can use that $\R_{212}=A_2\R_{21}$ and $\R_{21}=x_3$. Hence,
$$\R_{212}=A_2(x_3)=1+a_{12}x_2x_3=1+a_{12}x_1^2$$
The middle equation is obtained using the formula $A(y_1)=1+a_{12}y_1y_2+\ldots$ from
Section \ref{s.Gysin} and the
observation that all symmetric polynomials in $x_2$ and $x_3$ of degree greater than 2
vanish modulo $S$.
\paragraph{\bf Computing products of the Bott-Samelson resolutions.}
Let us for instance compute $Z_{12}Z_{21}$. First, we do it by hand.

Denote by $\om_1$, $\om_2$ the fundamental weights of $SL_3$.
Applying Proposition \ref{p.bundle}(2) to $X_{121}=X$ we get
$$L(\lambda)=L(X_{21})^{(\lambda,\g_2)}\otimes L(X_{12})^{(\lambda,\g_1)}.$$
Hence, $c_1(L(\om_1))=X_{12}$ and $c_1(L(\om_2))=X_{21}$. (Note that if we instead
applied Proposition \ref{p.bundle}(1)  to $R_{121}$, we would
obtain the more complicated expression
$$r_{121}^*L(\lambda)=L(R_{21})^{(\lambda,\g_2)}\otimes L(R_{11})^{(\lambda,\g_1+\g_2)}\otimes L(R_{12})^{(\lambda,\g_1)},$$
which does not allow us to express $X_{12}=R_{12}$ as the Chern class of the line bundle $L(\l)$ on $R_{121}$.)

Hence, $$Z_{12}Z_{21}=c_1(L(\om_1))Z_{21}={r_{21}}_*c_1(r_{21}^*L(\omega_1))$$ by the projection formula:
$c_1(L(\lambda))\cdot Z_I={r_I}_*c_1({r_I}^*L(\lambda)).$
We now apply Proposition \ref{p.bundle}(1) to $R_{21}$ and $L(\omega_1)$ and get
$r_{21}^*L(\omega_1)=L(R_1)\otimes L(R_2).$
Using the formal group law  we compute
$c_1(L(R_1)\otimes L(R_2))=R_1+R_2-[\P^1]R_e.$
Finally, we use that ${r_J}_*[R_J]=Z_J$ and get that
$$Z_{12} Z_{21}=Z_1+Z_2-[\P^1]Z_e.$$

Similarly, we can easily compute the following products:
$$Z_{12}Z_{12}=Z_2; \quad Z_{21}Z_{21}=Z_1$$
$$Z_{12}Z_1=Z_{21} Z_2=Z_e, \quad Z_{12}Z_2=Z_{21}Z_1=0,$$
which in particular gives us another way to compute polynomials $\R_I$.

So the only product that differs from the analogous product in the Chow ring case is
the product $Z_{12}Z_{21}$.

We now compute the product $Z_{12}Z_{21}$  using formula (5.1).
We have $Z_{12}=c_1(L(\om_1))$ by Proposition \ref{p.bundle}(1).
Hence, according to formula (5.1)
$$Z_{12}Z_{21}=c_1(L(\om_1))Z_{21}=b_1(\om_1)Z_2+b_2(\om_1)Z_1+b_{21}(\om_1)Z_{e},$$
where $b_1$, $b_2$ and $b_{21}$ are the constant terms in $A_1^*(c_1(L(\om_1)))$,
$[A_2^*s_1](c_1(L(\om_1)))$ and $[A_2^*A_1^*](c_1(L(\om_1)))$, respectively. We already know
that $b_1(\l)=(\l,\g_1)$ and $b_2(\l)=(\l,s_1\g_2)$. It remains to compute $b_{21}(\l)$.
First, by using that $L(\l)=L(s_1\l)\otimes L(\g_1)^{(\l,\g_1)}$ and the formal group law  we write
$$A_1^*(c_1(L(\l)))=\frac{c_1(L(\l))-c_1(L(s_1\l))}{c_1(L(\g_1))}=$$
$$=(\l,\g_1)+a_{11}(\l,\g_1)[c_1(L(s_1\l))+\frac{(\l,\g_1)-1}{2}c_1(L(\g_1))]+\mbox{terms of } \deg\ge2$$
Hence,
$$[A_2^*A_1^*](c_1(L(\l)))=a_{11}(\l,\g_1)A_2^*[c_1(L(s_1\l))+\frac{(\l,\g_1)-1}{2}c_1(L(\g_1))]+\mbox{terms of }\deg\ge1=$$
$$=a_{11}(\l,\g_1)[(\l,s_1\g_2)-\frac{(\l,\g_1)-1}{2}]+\mbox{terms of }\deg\ge1, $$
and $b_{21}=a_{11}(\l,\g_1)[(\l,s_1\g_2)-\frac{(\l,\g_1)-1}{2}]$. We get
$$c_1(L(\l))Z_{21}=(\l,\g_1)Z_2+(\l,s_1\g_2)Z_1+a_{11}(\l,\g_1)[(\l,s_1\g_2)-\frac{(\l,\g_1)-1}{2}]Z_{e}.$$
In particular, $c_1(L(\om_1))Z_{21}=Z_2+Z_1+a_{11}Z_{e}$ (which coincides with
the answer we have found above by hand).

Finally, note that it takes more work to compute $c_1(L(\l))Z_{21}$ using the
algorithm in \cite{BE} because apart from certain formal group law  calculations (which are more
involved than the calculations we used to find $b_{21}$) one has also to compute the
products $R_1^2$ and $R_2^2$ in $CH^*(R_{21})$.

\section{Appendix: Complex realization for cellular varieties}

We will now prove the following result stated in Section 2:

\begin{thm}
For any smooth cellular variety $X$ over $k$
and any embedding $k \to \C$,
the complex geometric realization functor of $\L$-algebras
$r:\Omega^*(X) \to MU^*(X(\C)^{an})$ is an isomorphism.
\end{thm}
\begin{proof}
Recall (see above) that the geometric realization functor
coincides with the map given by the universal property
of $\Omega^*$, and that both sides are freely generated
by (resolutions of the closures of) the cells. Thus
it suffices to show that it is an isomorphism if we pass
to the induced morphism after taking $\otimes_{\L} \Z$
on both sides, which we denote by $r'$.
Now by a theorem of Totaro \cite[Theorem 3.1]{To}
(compare also \cite[Remark 1.2.21]{LM}),
for cellular varieties the classical cycle class map
$c:CH^*(X) \to H^*(X(\C)^{an})$ (which is an isomorphism for cellular varieties
$X$) defined using fundamental classes and Poincar\'e duality
(see e. g. \cite[section A.3]{Ma}) factors as
$CH^*(X) \to MU^*(X(\C)^{an}) \otimes_{\L} \Z \cong H^*(X(\C)^{an})$,
and the left arrow in this factorization
is given by first taking any resolution of singularities
of the algebraic cycle and then applying $(\C)^{an}$.
We also have a morphism $q: \Omega^*(X) \to CH^*(X)$ which induces
an isomorphism $q':\Omega^*(X) \otimes_{\L}\Z \to CH^*(X)$
by Levine-Morel \cite[Theorem 1.2.19]{LM}
and corresponds to resolution of singularities
\cite[Section 4.5.1]{LM}. Putting everything together,
we obtain a commutative square

$\xymatrix{
\Omega^*(X) \otimes_{\L}\Z \ar[d]^{\cong}_{q'} \ar[r]^{r'} &
MU^*(X(\C)^{an}) \otimes_{\L}\Z \ar[d]^{f'}_{\cong} \\
CH^*(X) \ar[r]_{c}^{\cong} & H^*(X(\C)^{an})
}$

with the vertical maps and $c$ being isomorphisms,
which finishes the proof.
\end{proof}

\end{document}